\documentclass[11pt,a4paper]{amsart}
\usepackage[latin1]{inputenc}
\usepackage[english]{babel}
\usepackage[T1]{fontenc}
\usepackage{amsmath,amssymb}
\usepackage{stmaryrd}
\usepackage{enumerate}
\usepackage{amsthm}
\usepackage{mathrsfs}
\usepackage{mathabx}
\usepackage{geometry}
\usepackage{mathtools}
\usepackage{graphicx}
\usepackage{array}
\usepackage[draft,english]{fixme}
\usepackage[all,line]{xy}
\usepackage[usenames,dvipsnames,svgnames,table]{xcolor}
\usepackage{xfrac}
\mathtoolsset{showonlyrefs}
\newtheorem{thm}{Theorem}[section]
\newtheorem{prop}[thm]{Proposition}
\newtheorem*{prop nn}{Proposition}
\newtheorem*{thm nn}{Theorem}
\newtheorem{lemma}[thm]{Lemma}
\newtheorem*{lemma nn}{Lemma}

\newtheorem*{conj nn}{Conjecture}
\theoremstyle{definition}
\newtheorem{defin}[thm]{Definition}
\newtheorem*{defin nn}{Definition}

\theoremstyle{remark}
\newtheorem{rem}[thm]{Remark}
\newtheorem*{rem nn}{Remark}
\newtheorem{eks}[thm]{Example}
\newtheorem*{eks nn}{Example}
\newtheorem{claim}[thm]{Claim}
\newtheorem*{claim nn}{Claim}

\newtheorem*{cor nn}{Corollary}

\newtheorem*{exercise nn}{Exercise}

\newcommand{\iso}{\simeq}
\newcommand{\isomap}{\stackrel{\sim}{\to}}

\newcommand{\Z}{\mathbb{Z}}
\newcommand{\N}{\mathbb{N}}

\newcommand{\E}{\mathbb{E}}

\newcommand{\Acal}{\mathcal{A}}
\newcommand{\Bcal}{\mathcal{B}}
\newcommand{\Ccal}{\mathcal{C}}
\newcommand{\Ocal}{\mathcal{O}}

\newcommand{\Fcal}{\mathcal{F}}
\newcommand{\Dcal}{\mathcal{D}}

\newcommand{\Hcal}{\mathcal{H}}

\newcommand{\Ical}{\mathcal{I}}
\newcommand{\Jcal}{\mathcal{J}}

\newcommand{\Ebold}{\mathbb{E}}
\newcommand{\Ar}{\operatorname{Ar}}
\newcommand{\colim}{\operatorname{colim}}

\newcommand{\Coh}{\operatorname{Coh}}
\newcommand{\Hom}{\operatorname{Hom}}

\newcommand{\id}{\operatorname{Id}}

\newcommand{\Map}{\operatorname{Map}}
\newcommand{\For}{\operatorname{For}}

\newcommand{\red}{\operatorname{red}}

\newcommand{\lrhom}{\mathbin{\rotatebox[origin=c]{56}{$\lozenge$}}}
\newcommand{\rrhom}{\mathbin{\reflectbox{\rotatebox[origin=c]{56}{$\lozenge$}}}}
\numberwithin{equation}{section}
\usepackage{tikz-cd}
\usetikzlibrary{decorations.markings}
\tikzset{degil/.style={
        decoration={markings,
            mark= at position 0.5 with {
                \node[transform shape] (tempnode) {$\backslash$};}},
        postaction={decorate}}}

\begin{document}

\title{Colored DG-operads and homotopy adjunction for DG-categories}
\author{Sergey Arkhipov}
\author{Tina Kanstrup}
\address{S.A. Matematisk Institut, Aarhus Universitet, Ny Munkegade, DK-8000 , Aarhus C, Denmark, email: hippie@math.au.dk} 
\address{T.K. Matematisk Institut, Aarhus Universitet, Ny Munkegade, DK-8000 , Aarhus C, Denmark, email: tina.kanstrup@mail.dk}

\maketitle

\begin{abstract}
Generalizing the approach to pseudo monoidal DG-categories as certain colored non-symmetric DG-operads, we introduce a certain relaxed notion of a category enriched in DG-categories. We construct model structures on the category of colored non-symmetric DG-operads and on the category of DGCat-enriched categories with a fixed set of objects. This allows us to talk about strong homotopy maps in both settings. We discuss the notion of a strong homotopy monad in a  DG-category and a notion of strong homotopy adjunction data for two DG-functors.
\end{abstract}


\section{Introduction}

This is the first paper in a series whose goal is to define homotopy adjunction for DG-categories and prove homotopy descent for them. This notion is native to the language of infinity-categories. However, the language of infinity-categories is fairly involved and for many applications the much simpler language of DG-categories is sufficient.

\subsection{Classical adjunction}

In the classical setting Schanuel and Street \cite{SS} showed that the data of adjoint functors between two categories can be repackaged as a functor between two 2-categories. For us a 2-category is a category (strictly) enriched in the category of categories and a 2-functor is an enriched functor. Let $k$ be a field. We now define the $k$-linear versions of the categories involved

\begin{defin}
\begin{enumerate}
\item The category $\Delta_k$ is the $k$-linear category whose objects are ordered sets $(n):=(1, \dots, n)$ for $n \in \Z_{> 0}$ and $(0)=\emptyset$. The morphisms are $k$-linear combinations of ordered maps.
\item Set $\nabla_k :=(\Delta_k)^{\text{op}}$. It is the $k$-linear category whose objects are ordered sets $[n]:=(0,1, \dots, n)$ for $n \in \Z_{\geq 0}$ and morphisms are $k$-linear combinations of ordered maps preserving first and last element.
\item The category $\lrhom_k$ has objects non-empty ordered sets $(n]=(0,1, \dots, n)$ and morphisms are $k$-linear combinations of ordered maps preserving last element.
\item The category $\rrhom_k$ has objects non-empty ordered sets $[n)=(0,1, \dots, n)$ and morphisms are $k$-linear combinations of ordered maps preserving first element.
\end{enumerate}
\end{defin}

\begin{rem}\label{StreetRem}
$\Delta_k$ has a monoidal structure given by taking the disjoint union $(n) \cdot (m)=(n+m)$. The unit in $\Delta_k$ is (0) and $(n)=(1)^{ n}$. Similarly, $\nabla_k$ has a monoidal structure given by connected union $[n] \otimes [m]=[n+m]$ with unit $[0]$. We also have actions
\begin{gather} \Delta_k \times \lrhom_k \to \lrhom_k, \qquad (n), (m] \mapsto (n+m] \\
\rrhom_k \times \nabla_k \to \rrhom_k, \qquad [n), [m] \mapsto [n+m)
\end{gather}
notice that $(n]=(1)^{n} \otimes (0]$ and $[n)=[0) \otimes [1]^n$. Moreover, we have maps given by connected union by end and beginning, and by disjoint union.
\begin{gather}
\lrhom_k \times \rrhom_k \to \Delta_k, \qquad (n], [m) \mapsto (n+m+1),\\
\rrhom_k \times \lrhom_k \to \nabla_k, \qquad [n), (m] \mapsto [n+m+1].
\end{gather}
\end{rem}

Using these categories we can construct the free adjunction 2-category:

\begin{defin}\label{FreeAdj}
The free adjunction 2-category Adj is the 2-category with two objects $\{0,1\}$ and Adj$_{0,0}=\Delta_k$, Adj$_{0,1}=\lrhom_k$, Adj$_{1,1}=\Delta_k^{\text{op}}$ and Adj$_{1,0}=\lrhom_k^{\text{op}}$.
\end{defin}

Let $A_1$ and $A_2$ be DG-categories. We collect the functors between them in a 2-category $\text{Fun}_{A_1,A_2}$ with two objects $\{0,1\}$ and $\Hom_{\text{Fun}_{A_1,A_2}}(i,j):=\text{Fun}(A_i, A_j)$. The data of a 2-functor $\Fcal:\text{Adj} \to \text{Fun}_{A_1,A_2}$ is given by four functors each of which is determined by its value on 1. The data of a pair of adjoint functors define four such functors in the following way.
\begin{align}
&\Fcal_{0,0}: \Delta_k \to \text{Fun}(A_1,A_1), \qquad (1) \mapsto G \circ F\\
&\Fcal_{0,1}: \lrhom_k \to \text{Fun}(A_1,A_2), \qquad (1] \mapsto F\\
&\Fcal_{1,0}: \rrhom_k \to \text{Fun}(A_2,A_1), \qquad [1) \mapsto G\\
&\Fcal_{1,1}: \nabla_k \to \text{Fun}(A_2,A_2), \qquad [1] \mapsto F \circ G
\end{align}

\begin{thm}\cite{SS}
The data of a pair of adjoint functors $F : A_1 \leftrightarrows A_2 : G$ is equivalent to the data of a 2-functor $\Fcal:\text{Adj} \to \text{Fun}_{A_1,A_2}$ which is identity on 0-morphisms.
\end{thm}

When $A_1$ and $A_2$ are ordinary categories we have the Barr-Beck theorem

\begin{thm}[Barr-Beck]
Let $T$-mod be the category of modules over the monad $T=G \circ F$. Assume that $G$ commutes with colimits. Then there is an equivalence of categories $T$-mod $\iso A_2$ if and only if $G$ is conservative.
\end{thm}

\subsection{Homotopy version}
To pass to the homotopy setting we need quasi-isomorphisms to become invertible. A natural way to achieve this is to replace usual DG-functors by A$_\infty$-functors. For this we need to work with colored non-symmetric operads.

\begin{defin nn}
\begin{enumerate}[(i)]
\item Let $\E$ be a set and set $k_{\E}:=\oplus_{s \in \E} k e_s$, where $e_s$ are basis elements with $e_s e_t=\delta_{s,t} e_s$. An $\N$-collection $V$ over $k_\E$ is a collection of complexes $V=\{V(n) \mid n \geq 1\}$ such that $V(n) \in k_\E^n$-mod-$k_\E$. I.e. it has a decomposition $V(n)=\bigoplus_{s_1, \dots, s_n,t \in \E} V(s_1, \dots , s_n,t)$.
\item A morphism $f$ between $\N$-collections $A$ and $B$ is a collection \[\{ f(n) \in \Hom_{k_\E^n \text{-mod-} k_\E}(A(n), B(n)) \mid n \geq 1\}.\]
This category is denoted by $\N$-col$_\E$ and it has a monoidal structure given by
\[ (V \odot W)(n):=\bigoplus_{m \in \N, n_1+\dots + n_m=n} W(m)  \otimes_{k_\E^m}  \bigl[V(n_1) \boxtimes_{k} \cdots \boxtimes_{k} V(n_m))\bigr].\]
\item A colored non-symmetric DG-operad with colors $\E$ is a unital associative algebra in $\N$-col$_\E$.
\item A DG-operad is unital if for any $n \in \N$ and tuple $m_1, \dots, m_n$ with $0 \leq m_i \leq 1$ the morphism of complexes $A(n) \otimes_{A(1)^n} (A(m_1) \otimes \cdots \otimes A(m_n)) \to A(m_1 + \dots + m_n)$ is an isomorphism.
\end{enumerate}
\end{defin nn}

A symmetric monoidal DG-category $A$ defines an operad by setting $A(s_1, \dots, s_n, t) := \Hom_A(s_1 \otimes \cdots \otimes s_n,t)$. Tabuada \cite{Ta} constructed a model category structure on the category of DG-categories. In a similar way we construct one for operads and for unital operads.

\begin{thm nn}[\ref{ModelCatOperad} and \ref{ModStructUOper}] The category of (unital) colored non-symmetric DG-operads has a cofibrantly generated model category structure in which all objects are fibrant.
\end{thm nn}

For symmetric operads there is a notion of a Bar and a Cobar construction (see \cite{GK}). We adapt their construction to our setting and prove that they are adjoint and that Cobar(Bar$(P)$) of an operad $P$ is a cofibrant replacement. Hence, an A$_\infty$-functor between two operads is a functor between their Bar constructions.

The homotopy version of the category of 2-categories used by Schanuel and Street needs to be enriched in DG-categories. We consider only 2-categories with a fixed set of objects $I$ and only 2-functors which are identity on the set of objects. The set $\E$ is then replaced with a collection of sets $\E=\{\E_{ij}\}_{i,j \in I}$. One may think of elements of $I$ as 0-morphisms and elements of $\E$ as 1-morphisms. Taking source or target defines two maps $s,t: \E \to I$. A path is a sequence of morphisms $\text{Path}_{\E,I}(n) :=\E \times_I \cdots \times_I \E$ ($n$ factors) and a cell is a morphism with a path going between the same elements $\text{Cell}_{\E,I}=\sqcup_{n} \text{Path}_{\E,I}(n) \times_{I \times I} \E$. Our version of 2-morphisms will correspond to homotopies between them. Concatenation of paths gives a natural structure of an operad in sets $\text{Cell}^n \times_{\text{Path}^n} (\text{Cell}^{m_1} \times_I \dots \times_I \text{Cell}^{m_n})\to \text{Cell}^{m_1+ \dots + m_n}$ with $(t_1, \dots, t_n;t), (s_1; t_1), \dots, (s_n;t_n) \mapsto (s_1, \cdots, s_n; t)$. This can be rewritten on the level of rings. Set $k_I:=\bigoplus_{r \in \E} k e_r$,  $k_\E:=\bigoplus_{i,j \in I} k_{\E_{i,j}}$, $\Ocal({\text{Path}^n}) :=k_\E \otimes_{k_I} \dots \otimes_{k_I} k_\E$, and $\Ocal(\text{Cell}^n)=\Ocal(\text{Path}^n) \otimes_{k_I \times k_I} k_\E$. Pulling back along the concatenation map gives a morphism
\begin{align} \label{OCellintro}
\Ocal(\text{Cell}^{m_1+ \dots + m_n}) \to  \Ocal(\text{Cell}^n) \otimes_{\Ocal(\text{Path}^n)} (\Ocal(\text{Cell}^{m_1}) \otimes_{k_I} \dots \otimes_{k_I} \Ocal(\text{Cell}^{m_n}))
\end{align}

\begin{defin nn}[\ref{NseqIE} and \ref{2CatI}]
\begin{enumerate}[(i)]
\item The category $\N \text{-seq}_{I, \E}$ has objects collections of complexes $\{A(n) \mid n \geq 1\}$ with $A(n)=\bigsqcup_{i,j \in I, t \in \E_{ij}, s \in \text{Path}^n(i,j)} A_{i,j}(s;t) \in \Ocal(\text{Cell}^n)\text{-mod}$.
\item The category $\N \text{-seq}_{I, \E}$ has a product $\odot$ given by 
\[(A \odot B)(m) := \bigoplus_{n \in \N, m_1 + \dots m_n=m} B(n) \otimes_{\Ocal({\text{Path}^n})} (A(m_1) \boxtimes_{k_I} \dots \boxtimes_{k_I} A(m_n)).\]
\item A category in 2-Cat$_I$ is a unital associative algebra in $\N-\text{seq}_{I, \E}$ with $A(n)  \otimes_{\Ocal({\text{Path}^n})} (A(m_1) \boxtimes_{k_I} \dots \boxtimes_{k_I} A(m_n)) \to A(m_1 + \dots +m_n)$
a morphism of $\Ocal(\text{Cell}^{m_1 + \dots m_n})$-modules via \eqref{OCellintro}. 
\item A morphism in 2-Cat$_I$ is the data $F=\{F_\E, F(n), n \geq 1\} : (I, \E_A, A) \to (I,\E_B, B)$. Here $F_\E : \E_A \to \E_B$ is a morphism. It induces a map of rings $\Ocal(\text{Cell}^n_A) \to \Ocal(\text{Cell}^n_B)$ and we require that this map is compatible with \eqref{OCellintro}. For $n \geq 1$ the $F(n)$ are morphisms of $\Ocal(\text{Cell}^n_A)$-modules $F(n) : A(n) \to B(n)$ satisfying composition.
\item A unital 2-category $A$ is an object in 2-Cat$_I$ satisfying that for any $n \in \N$ and tuple $m_1, \dots, m_n$ with $0 \leq m_i \leq 1$ the morphism of complexes $A(n) \boxtimes_{\Ocal(\text{Path}^n)} (A(m_1) \boxtimes \cdots \boxtimes A(m_n)) \to A(m_1 + \dots + m_n)$
is an isomorphism. A unital morphism is a morphism $F$ in 2-Cat$_I$ such that $F(\id_i)=\id_i$ for all $i \in I$.
\end{enumerate}
\end{defin nn}

\begin{thm nn}[\ref{2CatIModelStruct} and \ref{2CatIUModelStruct}]
The category 2-Cat$_I$ (and its unital version) has a cofibrantly generated model category structure in which all objects are fibrant.
\end{thm nn}

The model structure is a generalization of the one for operads. Likewise, we generalize the Bar and Cobar constructions and prove that for any object $A$ in 2-Cat$_I$ Cobar(Bar$(A)$) is a cofibrant replacement. The linearized adjunction category Adj from Schanuel and Street naturally defines an object in 2-Cat$_{\{0,1\}}$. For DG-categories $A_1,A_2$ we define the A$_\infty$ version $\text{DGFun}_{\infty}(A_1,A_2) \in 2\text{-Cat}_I$ of Fun$_{A_1,A_2}$ with 2-morphisms being coherent natural transformation of DG-functors.

\begin{defin nn}[\ref{HomAdjDef}]
Let $A_1,A_2$ be cofibrant DG-categories.
A homotopy adjunction is an A$_\infty$-morphism Cobar(Bar(Adj)) $\to \text{DGFun}_\infty(A_1,A_2)$.
\end{defin nn}

Notice that Anno and Logvinenko in their recent paper \cite{AL} produced a certain canonical amount of higher data for two DG-functors between two DG-categories A and B (realized as bimodules) upgrading the fact that the corresponding functors on the homotopy categories are adjoint. We plan to compare their data with our universal adjunction data for DG-functors.

\subsection*{Relation to Barr-Beck theorem}
For $F$ and $G$ a pair of adjoint functors between ordinary categories $\Ccal_1$ and $\Ccal_2$ the composition $T = G \circ F$ is a monad on $\Ccal_1$. Under certain assumptions on $F$ and $G$ the Barr-Beck theorem states that the category of $T$-modules in $\Ccal_1$ is equivalent to $\Ccal_2$. The goal is to provide an analog of Barr-Beck theorem for a pair of DG-categories $A_1, A_2$ and a pair of DG-functors $F$ and $G$ which become adjoint on the level of the corresponding homotopy categories. Restricting the morphism in 2-Cat$_{\{0,1\}}$ to a morphism of DG-operads Cobar(Bar(Adj$_{(0,0)}$)) $\to \text{DGFun}_{\infty}(A_1,A_2)_{(0,0)}$ defines a homotopy monad. The goal is to define the category of strong homotopy modules over the monad in $A_1$ and to prove the homotopy Barr-Beck theorem: 

\begin{conj nn}
Under certain assumptions, for a pair of homotopy adjoint functors $F$ and $G$ between $A_1$ and $A_2$, the DG-categories $T$-$\operatorname{hoMod}(A_1)$ and $A_2$ are quasi-equivalent.
\end{conj nn}

This should be compared to the Barr-Beck-Lurie theorem for $(\infty,1)$-categories in \cite{Lu}.

\subsection{Acknowledgements}
The authors would like to thank Alexander Efimov, Timothy Logvinenko, Dasha Poliakova, Boris Shoikhet, and Sebastian \O rsted for stimulating discussions.

This paper was written while the second author was a postdoc at the Max Planck Institute for Mathematics. Most of the work was done while the first author was also visiting. Both authors would like to express their gratitude to MPIM for the invitations and for excellent working conditions.

\section{DG-categories}
In what follows, all DG-categories and operads are considered to be small.
We collect the categorical data into that of certain DG-algebras and DG-modules. Thus, we start by considering the corresponding setup for DG-categories. Let $k$ be a field and $\E$ be a set of objects. Set $k_{\E}:=\oplus_{s \in \E} k e_s$, where $e_s$ are basis elements with $e_s e_t=\delta_{s,t} e_s$. This is a non-unital associative algebra. Let $k_\E$-mod-$k_\E$ be the category of complexes of vector spaces with two commuting actions
\[k_\E \otimes V \otimes k_\E \to V, \qquad V \in \text{Com(Vect)} \]
such that for all $m \in V$ there exist only finitely may $s,t \in \E$ such that $s \cdot m \neq 0$ and $m \cdot t \neq 0$. Moreover, we require that $\sum_s e_s$ acts as identity.

\begin{rem}
For $s,t \in \E$ we write $V(s,t):=sVt$. It follows that for all $V \in k_\E$-mod-$k_\E$
\[ V=\bigoplus_{s,t \in \E}V(s,t). \]
\end{rem}

The category $k_\E$-mod-$k_\E$ has a monoidal structure given by
\[ (V \otimes W)(s,t):=\bigoplus_{u \in \E} V(u,t) \otimes W(s,u), \qquad V,W \in k_\E\text{-mod-}k_\E. \]
The unit object is the regular bimodule placed in degree zero $\underline{k}_\E=\oplus_{s,t} \underline{k}_\E(s,t)$.

\begin{defin}
A DG-category with set of objects $\E$ is a unital associative algebra in the monoidal category $k_\E$-mod-$k_\E$. Thus, we have the maps
\[ A \otimes A \to A, \qquad \text{given by } A(u,t) \otimes A(s,u) \to A(s,t) \]
and $k_\E \to A$ given by $e_s \mapsto A(s,s)$. This forms a category denoted by DGCat$_\E(k)$ with morphisms being maps of unital associative algebras.
\end{defin}

The forgetful functor
\[\text{Obl}: \text{DGCat}_\E(k) \to k_\E\text{-mod-}k_\E \]
has a left adjoint denoted by Free$_\E$
\[\text{Free}_\E(V):=\bigoplus_{n \in \N} V \otimes \cdots \otimes V. \]

\subsection{DG-functors}
Let $\E_1$ and $\E_2$ be sets of objects and let $f: \E_1 \to \E_2$ be a map. This gives rise to a map of algebras $f: k_{\E_1} \to k_{\E_2}$ and so a pullback functor
\begin{gather}
f^* : k_{\E_2}\text{-mod-} k_{\E_2} \to k_{\E_1}\text{-mod-} k_{\E_1},\\
f^*(V)(s,t):=V(f(s), f(t))
\end{gather}

Observe that
\begin{align}
f^*(V \otimes W)(s,t) & = V \otimes W(f(s), f(t))\\
&=\bigoplus_{u \in \E_2} V(u,f(t)) \otimes W(f(s),u),\\
f^*(V) \otimes f^*(W)(s,t) &=\bigoplus_{v \in \E_1} V(f(v), f(t)) \otimes W(f(s), f(v)).
\end{align}
Hence, there is a map
\[ f^*(V) \otimes f^*(W) \to f^*(V \otimes W) \]
The functor is lax monoidal. Notice also that there is a natural map $k_{\E_1} \to f^*(k_{\E_2})$ given by $e_s \mapsto e_{f(s)}$. Let $k_{\E_2} \to V$ be a unit map in $k_{\E_2}\text{-mod-} k_{\E_2}$. Then the composition $k_{\E_1} \to f^*(k_{\E_2}) \to f^*(V)$ is a unit map in $k_{\E_1}\text{-mod-} k_{\E_1}$. Hence, $f^*$ takes unital associative algebras to unital associative algebras. In this way we reformulate the usual definitions of DG-categories and DG-functors as certain associative algebras and certain associative algebra maps.

\begin{defin}
Let $A_1 \in \text{DGCat}_{\E_1}(k)$ and $A_2 \in \text{DGCat}_{\E_2}(k)$. A DG functor $f: A_1 \to A_2$ is a pair of a map $f: \E_1 \to \E_2$ and $F: A_1 \to f^*(A_2)$ a map in DGCat$_{\E_1}(k)$.
\end{defin}

\section{Colored non-symmetric DG-operads}

In this section we recall the definition of the category of colored non-symmetric DG-operads, which is a multicategory version of DG-categories.

\begin{defin}
An $\N$-collection $V$ over $k_\E$ is a collection of complexes $V=\{V(n) \mid n \geq 1\}$ such that $V(n) \in k_\E^n$-mod-$k_\E$. I.e. it has a decomposition
\[ V(n)=\bigoplus_{s_1, \dots, s_n,t \in \E} V(s_1, \dots , s_n,t). \]
A morphism $f$ between $\N$-collections $A$ and $B$ is a collection $\{ f(n) \in \Hom_{k_\E^n \text{-mod-} k_\E}(A(n), B(n)) \mid n \geq 1\}$. This category is denoted by $\N$-col$_\E$ and it has a monoidal structure given by
\[ (V \odot W)(n):=\bigoplus_{\substack{m \in \N, \\ n_1+\dots + n_m=n}} W(m)  \otimes_{k_\E^m}  \bigl[V(n_1) \boxtimes_{k} \cdots \boxtimes_{k} V(n_m))\bigr].\]
\end{defin}

\begin{defin}
\begin{enumerate}
    \item A colored non-symmetric DG-operad with colors $\E$ is a unital associative algebra in $\N$-col$_\E$. We denote the category of these by DG-Oper$_\E$.
    \item  A colored non-symmetric DG-cooperad with colors $\E$ is a counital coassociative coalgebra in $\N$-col$_\E$. We denote the category of these by DG-CoOp$_\E$
\end{enumerate}
\end{defin}

In this paper all operads are non-symmetric so from now on we drop writing non-symmetric. We warn the reader that many authors use the word operad to refer to \emph{symmetric} operads (which are required to be invariant under the action of the symmetric group).

\begin{eks}
A monoidal DG-category $A$ defines a colored DG-operad by setting 
\[A(s_1, \dots, s_n, t) := \Hom_A(s_1 \otimes \cdots \otimes s_n,t).\]
\end{eks}

There is a pair of adjoint functors
\[ (1) : \text{DG-Oper}_\E \rightleftarrows \text{DG-Cat}_\E : \text{Triv}\]
where $(1)$ is the forgetful functor taking $A$ to the DG-category $A(1)$ with the same objects as $A$ and $\Hom_{A}(s,t) :=A(s,t)$. The DG-operad Triv$(B)$ is defined by $B(s,t)=\Hom_B(s,t)$ and $B(s_1,\dots, s_n,t)=0$ for $n>1$.

\subsection{DG-functors for operads}

One can define DG-functors for operads with different sets of objects in the same manner as for DG-categories. Let $f : \E_1 \to \E_2$ be a map. This defines a functor
\begin{gather}
f^* : \N\text{-col}_{\E_1} \to \N\text{-col}_{\E_2}\\
f^*(V)(s_1, \dots, s_n, t):=V(f(s_1), \dots, f(s_n), f(t)).
\end{gather}
Observe that
\begin{align}
    f^*(V \odot W)&(s_1, \dots, s_n, t) =(V \odot W)((f(s_1), \dots, f(s_n), f(t))\\
    &=\;\smash[b]{\smashoperator{\bigoplus_{\substack{m \in \N, n_1+\dots + n_m=n, \\u_1, \dots, u_m \in \E_2}}}}\; (V(f(s_1), \dots, f(s_{n_1}), u_1) \otimes \dots \otimes V(f(s_{n_1 + \dots +n_{m-1}}), \dots, f(s_n), u_m)) \\
    & \hspace{3cm} \otimes W(u_1, \dots, u_m, f(t)).
\end{align}
\begin{align}
f^*(V) &\odot f^*(W)(s_1, \dots, s_n, t)\\
&=\;\smash[b]{\smashoperator{\bigoplus_{\substack{m \in \N, n_1+\dots + n_m=n, \\v_1, \dots, v_m \in \E_1}}}}\;(f^*(V)(s_1, \dots, s_{n_1}) \otimes \dots \otimes f^*(V)(s_{n_1+\cdots+n_{m-1}}, \dots, s_n, v_m)) \\
& \hspace{3cm}\otimes f^*(W)(v_1, \dots, v_m, t)\\
 &=\;\smash[b]{\smashoperator{\bigoplus_{\substack{m \in \N, n_1+\dots + n_m=n, \\u_1, \dots, u_m \in \E_2}}}}\; (V(f(s_1), \dots, f(s_{n_1}), f(v_1)) \otimes \dots \otimes V(f(s_{n_1 + \dots +n_{m-1}}), \dots, f(s_n), f(v_m))) \\
    & \hspace{3cm} \otimes W(f(v_1), \dots, f(v_m), f(t)).
\end{align}
Just like for DG-categories there is a map given by inclusion
\[ f^*(V) \odot f^*(W) \to f^*(V \odot W)\]
Just like in the DG-category setting the functor is lax monoidal. A unit map $k_\E \to V$ maps into $V(1)$ so the fact that $f^*$ takes unital associative algebras to unital associative algebras now follows from the DG-category setting. We also have a lax monoidal functor given by projection
\[ f^*(V \odot W) \to f^*(V) \odot f^*(W)\]
Let $V \to k_{\E_2}$ be a counit map. Then the composition map $f^*(V) \to f^*(k_{\E_2}) \to k_{\E_1}$ is a counit map. Hence, we have the analogous definition

\begin{defin}
\begin{enumerate}
    \item Let $A_1 \in \text{DG-Oper}_{\E_1}(k)$ and $A_2 \in \text{DG-Oper}_{\E_2}(k)$. A morphism $A_1 \to A_2$ is a pair of a map $f: \E_1 \to \E_2$ and $F: A_1 \to f^*(A_2)$ a map in DG-Oper$_{\E_1}(k)$. 
    \item Let $A_1 \in \text{DG-Coop}_{\E_1}(k)$ and $A_2 \in \text{DG-Coop}_{\E_2}(k)$. A morphism $A_1 \to A_2$ is a pair of a map $f: \E_1 \to \E_2$ and $F: f^*(A_2) \to A_1$ a map in DG-Coop$_{\E_1}(k)$.
\end{enumerate}
Let DG-Oper$(k)$ (resp. DG-Coop$(k)$) denote the category with objects being objects in DG-Oper$_\E(k)$ (resp. DG-Coop$_\E(k)$) for some set $\E$, and morphisms the ones just defined.
\end{defin}

\section{Model category structure on colored DG-operads}

In this section we define a model category structure on the category of colored non-symmetric DG-operads. A similar model category structure for simplicial operads was constructed in a paper by Cisinski and Moerdijk \cite{CM}. A paper by Caviglia \cite{Ca} proves that, under certain conditions, the model structure on a monoidal model category can be transferred to a model structure on the category of  colored operads enriched over this category. Our paper is independent of \cite{Ca}. We prove the theorem

\begin{thm} \label{ModelCatOperad}
The category of colored DG-operads has a cofibratntly generated model structure.
\end{thm}

Note that the category of colored DG-Operads has pullbacks and small products so it has all small limits. Dually, it also has pushouts and small coproducts so it has all small colimits. We will need an explicit pushout construction for the proof of the main theorem \ref{ModelCatOperad} in section \ref{ProofMainThm}. The proof follows the same lines as Tabuada's construction \cite{Ta} of a model category structure on the category of DG-categories. We explicitly define a proposed set of generating cofibrations and generating trivial cofibrations and check that they define a model category structure by verifying the conditions in the following recognition theorem.

\begin{thm}\cite[Theorem 2.1.19]{Ho} \label{Ho}
Suppose $\Ccal$ is a category with all small colimits and limits. Suppose $W$ is a subcategory of $\Ccal$, and $I$ and $J$ are sets of maps of $\Ccal$. Then there is a cofibrantly generated model structure on $\Ccal$ with $I$ as the set of generating cofibrations, $J$ as the set of generating trivial cofibrations, and $W$ as the subcategory of weak equivalences if and only if the following conditions are satisfied.
\begin{enumerate}
\item The subcategory $W$ has the two out of three property and is closed under retracts.
\item The domains of $I$ are small relative to $I$-cell.
\item The domains of $J$ are small relative to $J$-cell.
\item $J$-cell $\subseteq W \cap I$-cof.
\item $I$-inj $\subseteq W \cap J$-inj.
\item Either $W \cap I$-cof $\subseteq J$-cof or $W \cap J$-inj $\subseteq I$-inj.
\end{enumerate}
\end{thm}

The model structure we construct is the following.

\begin{defin}\label{ModelStruct}
Let $F : \Acal \to \Bcal$ be a functor between DG-operads
\begin{enumerate}
\item The functor $F$ is a weak equivalence if
	\begin{enumerate}
	\item For all $n \in \N$ and $s_1, \dots, s_n, t \in \Ebold_\Acal$ the morphism of complexes
	\[ \Acal(s_1, \dots, s_n,t) \to \Bcal(F(s_1), \dots, F(s_n), F(t))\]
	is a quasi-isomorphism.
	\item The morphism $H^0(F(1)): H^0(\Acal(1)) \to H^0(\Bcal(1))$ is an equivalence of categories.
	\end{enumerate}
\item The functor $F$ is a fibration if
	\begin{enumerate}
	\item For all $n \in \N$ and $s_1, \dots, s_n,t$ the morphism of complexes
	\[ \Acal(s_1, \dots, s_n,t) \to \Bcal(F(s_1), \dots, F(s_n), F(t))\]
	is termwise surjective.
	\item The morphism $H^0(F(1)): H^0(\Acal(1)) \to H^0(\Bcal(1))$ is an isofibration, i.e. for any map $\psi$ such that $H^0(\psi)$ is an isomorphism in $H^0(\Bcal(1))$ there exists a map $\phi$ such that $H^0(\phi)$ is an isomorphism in $H^0(\Acal(1))$ and $F(\phi)=\psi$.
	\end{enumerate}
\item The functor $F$ is a cofibration if it has the left lifting property with respect to trivial fibrations.
\end{enumerate}
\end{defin}

\subsection{Generating cofibrations and generating trivial cofibrations}
Imitating Tabuada \cite{Ta} we define some specific colored DG-operads. The initial operad is denoted by $\emptyset$. Let $\Acal$ be the DG-category with one object $\{*\}$ and $\Acal(*,*)=k$. 

Let $X$ be a complex of $k$-modules. Write Ar$_m(X)$ for the colored operad with colors $\E=\{s_1, \dots, s_n, t\}$ and $\Ar_m(s_1, \dots, s_n, t)=X, \Ar_m(s_i,s_i)=k=\Ar_m(t,t)$, and all others zero. Let $S(n)$ be the complex with $k$ in degree n and zero elsewhere, and $D(n)$ the complex id:$k \to k$ in degrees $n-1$ and $n$ with zeros elsewhere. The standard map $S(n) \to D(n)$ induces a map of operads $\Ar_m(S(n)) \to \Ar_m(D(n))$. 

Define $\Hcal$ to be the DG-category with two objects 1 and 2 and whose morphisms are generated by $f \in \Hom^0(1,2)$, $g \in \Hom^0(2,1)$, $r_1 \in \Hom^{-1}(1,1)$, $r_2 \in \Hom^{-1}(2,2)$ and $r_{12} \in \Hom^{-2}(1,2)$ subject to relations $df=dg=0, dr_1=gf-\id_1, dr_2=fg-\id_2$ and $dr_{12}=fr_1-r_2f$. Observe that the category $H^0(\Hcal)$ is the category with two objects 1 and 2 and morphisms $\Hom(1,1)=k=\Hom(2,2)$, $\Hom(1,2)=k \cdot f$ and $\Hom(2,1)=k \cdot g$ and relations $fg=\id_2$ and $gf=\id_1$. Hence, it is the category with two objects and an isomorphism between them.

\begin{lemma}\label{HcalMap}
Let  $\Ccal$ be a DG-category. The existence of a functor $\Hcal \to \Ccal$ is equivalent to the existence of a morphism $f: c_1 \to c_2$ of degree 0 in $\Ccal$ with $df=0$ such that its cohomology class $[f] \in H^0(\Ccal)(c_1, c_2)$ is an isomorphism.
\end{lemma}
\begin{proof}
Since not all DG-categories allow cones we use the Yoneda embedding $\Ccal \hookrightarrow \Ccal^{\text{op}}$-mod and take cones in the module category. Recall that if $f:K^\bullet \to L^\bullet$is a map of complexes, then the cone of $f$ is the complex Cone$^\bullet(f) := K^{\bullet+1} \oplus L^\bullet$ with the differential $D:=(-d_K, f + d_L)$. The chain map $f$ is a homotopy equivalence if and only if Cone$^\bullet(f)$ is nullhomotopic. A nullhomotopy on Cone$^\bullet(f)$ is a map $H : \text{Cone}^\bullet(f) \to \text{Cone}^{\bullet-1}(f)$ satisfying $HD+DH=\id_\text{Cone}$. Write
\[H:=\begin{bmatrix} r_1 & g \\ r_{12} & r_2 \end{bmatrix} \]
for some maps $r_1:K^{\bullet+1} \to K^\bullet$, $r_2: L^\bullet \to L^{\bullet-1}$, $g:L^\bullet \to K^\bullet$ and $r_{12}: K^{\bullet+1} \to L^{\bullet-1}$. The condition $HD+DH=\id_\text{Cone}$ can then be rewritten as the following four equations
\begin{align}
&dr_1=r_1 d_K + d_K r_1 = gf-id_K,\\
&dg=gd_L- d_K g = 0,\\
&dr_{12}=r_{12}d_K - d_Lr_{12} = r_2 f + fr_1,\\
&dr_2=r_2 d_L +d_L r_2 =id_L-fg .
\end{align}
Hence, if $df=0$ then a nullhomotopy is equivalent to a morphism $\Hcal \to \text{Cone}(f)$.
\end{proof}

The following definition is a natural generalization of Tabuada.

\begin{defin}
\begin{enumerate}
\item The set $I$ of generating cofibrations is the set of maps $\{\emptyset \to \Acal, \Ar_m(S_{n-1})\to \Ar_m(D_n); m,n \in \N\}$.
\item The set $J$ of generating trivial cofibrations is the set of maps $\{\Acal \to \Hcal, \Ar_m(0) \to \Ar_m(D_n) ; m,n \in \N \}$.
\end{enumerate}
\end{defin}

The follow lemma shows that the definition of $J$ is compatible with the expected model category structure from \ref{ModelStruct}.

\begin{lemma}\label{FibrationRLP}
A functor is a fibration if and only if it has the right lifting property with respect to $J$.
\end{lemma}
\begin{proof}
First consider diagrams of the form
\[ \xymatrix{\Ar_m(0) \ar[d] \ar[r]^\alpha & P \ar[d]^F \\ \Ar_m(D_n) \ar[r]^\beta & Q}\]
Such a diagram is equivalent to a choice of $\beta(a) \in Q^n$. Hence, the right lifting property with respect to this morphism is equivalent to $F$ being surjective.

Let $F : P \to Q$ be a functor with he right lifting property with respect to $\Acal \to \Hcal$. A diagram induces a diagram on the underlying homotopy categories
\[ \xymatrix{H^0(\Acal) \ar[d] \ar[r] & H^0(P(1)) \ar[d]^{H^0(F(1))} \\ H^0(\Hcal) \ar[r] \ar@{-->}[ur] & H^0(Q(1))}\]
Since $H^0(\Hcal)$ is the category with two objects and an isomorphism between them the induced lifting property is equivalent to $H^0(F(1))$ being an isofibration. 

Assume now that $F$ is an isofibration. Given a diagram
\[ \xymatrix{\Acal \ar[d] \ar[r] & P \ar[d]^{F} \\ \Hcal \ar[r] \ar@{-->}[ur] & Q}\]
we need to construct a lifting. The diagram is given by $c \in \text{Ob}(P)$ and $a,b \in \text{Ob}(Q)$ with maps between them and $F(c)=a$. When passing to the homotopy categories the previous argument gives that $a \iso b$ in $H^0(Q(1))$. Since $F$ is an isofibration the exists $d \in \text{Ob}(H^0(P(1)))=\text{Ob}(P)$ such that $F(d)=b$ and an isomorphism $\phi : c \isomap d$ in $H^0(P(1))$. By lemma \ref{HcalMap} this is equivalent to providing a map $\Hcal \to P$ making the diagram commutative.
\end{proof}

\subsection{Proof of main theorem \ref{ModelCatOperad}}\label{ProofMainThm}
In this section we verify the conditions of the recognition theorem \ref{Ho}. The verification is broken down into a sequence of lemmas.

\begin{lemma} \label{Jcell}
$J$-cell $\subseteq W$.
\end{lemma}

We need to show that transfinite compositions of pushouts of the maps in $J$ are weak equivalences. For the proof we need the following lemma.

\begin{lemma}\label{FullyFaithful}
Let $\Ccal$ be a DG-category with one object and $\Dcal$ a DG-category with two objects (considered as operads) and $\alpha : \Ccal \to \Dcal$ a fully-faithful functor. Then for any pushout diagram of operads
\[ \xymatrix{\Ccal \ar[d]^{\alpha} \ar[r]^\gamma & P \ar[d]^\beta \\ \Dcal \ar[r] & Q}\]
the morphisms
\[ P^{n}(s_1, \dots, s_m,t) \to Q^{n}(\beta(s_1), \dots, \beta(s_m), \beta(t)) \qquad \forall s_1, \dots, s_m, t \in \text{Ob} P. \]
are quasi-isomorphisms.
\end{lemma}

\begin{proof}
The proof is identical to the proof of lemma 1.29 in \cite{CM} (they work in the setting of simplicial symmetric operads but the same proof works in our setting). They calculate the pushout explicitly. Since $\Ccal$ only have one object $\{0\}$ and $\Dcal$ has two objcets $\{s,t\}$ with $\alpha(0)=s$ the set of objects $\E_Q$ is the disjoint union $\E_P \sqcup \{t\}$. The morphisms $Q(c_1, \dots, c_n, c)$ depnds on where $t$ occurs. If $c_1,\dots c_n, c \in \E_P$ then $Q(c_1, \dots, c_n, c)=P(c_1, \dots, c_n, c)$ and $Q(t,t)=\Dcal(t,t)$. In mixed cases morphisms are represented by triples $h \cdot p \cdot (h_1, \dots, h_b)$, where $(h_1, \dots, h_n) \in \Dcal(1,0)^b$,  $p \in P(c_1, \dots,c_i, \gamma(0),c_{i+2} \dots \gamma(0), \dots,c)$ with $\gamma(0)$ occuring $b$ times, and $h \in \Dcal(0,1)$ if $c=t$ and trivial otherwise. With the following relations for all $k \in \Ccal(0,0)$
\begin{gather}
h \alpha(k) \cdot p \cdot (h_1, \dots, h_b)=h \cdot \gamma(k) p \cdot (h_1, \cdots, h_b)\\
h \cdot p \cdot (h_1, \dots, \alpha(k)h_i, \dots, h_b)=h \cdot p \circ_i \gamma(k) \cdot (h_1, \dots , h_b),
\end{gather}
where $\circ_i$ denotes inserting in the $i$'th place. Consider $h' \cdot p' \cdot (h'_1, \dots, h'_{b'})$ and $h \cdot p \cdot (h_1, \dots , h_b)$. Since $\alpha$ is fully-faithful there is a unique $k \in \Ccal(0,0)$ such that $\alpha(k)=h'_1 h \in \Dcal(0,0)$. We define the composition as
\[(h' \cdot p' \cdot (h'_1, \dots, h'_{b'})) \circ_{n+1} (h \cdot p \cdot (h_1, \dots , h_b))=h' \cdot (p' \circ_{n+1} \gamma(k)p) \cdot (h_1, \dots, h_b, h'_2, \dots, h'_{b'}) \]
The morphism $\beta$ is identity on $P(c_1, \dots, c_n, c)$ for $c_i,c \in \E_P$ so it is fully-faithful. See \cite{CM} for a proof that this defines the pushout.
\end{proof}

\begin{proof}[Proof of lemma \ref{Jcell}]
The class of weak equivalences is closed under transfinite compositions so it suffices to show that pushouts of $J$ are weak equivalences. Assume we have a pushout diagram
\[ \xymatrix{\Acal \ar[d] \ar[r]^\delta & P \ar[d]^\beta \\ \Hcal \ar[r] & Q}\]
We want to show that $\beta$ is a weak equivalence. Let $\Hcal_0$ be the category with one object and $\Hcal_0(*,*)=\Hcal(1,1)$. The pushout can be decomposed into two pushouts
\[ \xymatrix{\Acal \ar[d] \ar[r]^\delta & P \ar[d]^\phi \\ \Hcal_0 \ar[r] \ar[d]& P' \ar[d]^\pi \\ \Hcal \ar[r] & Q}\]
The functor $\Hcal_0 \to \Hcal$ is fully faithful so it follows from lemma \ref{FullyFaithful} that $\pi$ is a quasi-isomorphism on polyhoms. To check that $\phi$ is a weak equivalence we need the following observation.

\begin{claim}
The DG-category $\Hcal_0$ is the direct sum of $k$ and an exact complex $\bar{\Hcal}_0$.
\end{claim}
\begin{proof}
Write $\Hcal_0(0,0)=\bigoplus_{n \leq 0} \Hcal^n_0(0,0)$. Notice that
\[ H^0(\Hcal_0(0,0))=\Hcal^0_0(0,0)/ d\Hcal^{-1}(0,0) \iso k \]
The morphism $\Acal(0,0)=k \to \Hcal_0(0,0)$ is given by $k \cdot 1 \mapsto k \cdot \id$ fits into the short exact sequence
\[ k \to \Hcal_0(0,0) \stackrel{\epsilon}{\to} \Hcal^0_0(0,0)/d \Hcal^{-1}_0(0,0) \iso k \]
Hence, $\ker(\epsilon)=\bigoplus_{n<0} \Hcal^n(0,0) \oplus d \Hcal^{-1}(0,0)$ is a 2-sided ideal and an exact complex. Since $\Hcal_0(0,0) \iso k \oplus \ker(\epsilon)$ this finishes the proof of the claim.
\end{proof}

\begin{proof}
The pushout have $\E_{P'}=\E_P$ and morphisms are represented by by compositions of $h \cdot p$, where $h \in \Hcal_0(0,0)$ and $p \in P(c_1,\dots,c_n,\delta(0))$ with relations
\[h \beta(0) \cdot p=h \cdot \delta(0)p. \]
Composition is given by $(h \cdot p) \circ (h' \cdot p')=hh' \cdot pp'$. Hence, $P'=P \otimes \Hcal_0 /k \cdot \id \iso P \otimes \ker(\epsilon)$. It follows from the claim that $\phi$ is a quasi-isomorphism on polyhoms.
\end{proof}

Thus, $\phi$ is a quasi-isomorphism on polyhoms.

The second condition for a weak equivalence is in terms of the underlying DG-categories. Restricting a pushout diagram of operads to the underlying DG-categories gives a pushout diagram of DG-categories. Hence, it follows from the corresponding statement in Tabuada \cite[Lemme 2.2]{Ta}.

Assume we have a pushout diagram
\[ \xymatrix{\Ar_m(0)\ar[d] \ar[r]^\alpha & P \ar[d]^F \\ \Ar_m(D_n) \ar[r] & Q}\]
For $m \neq 1$ we have $\Ar_m(0)(1)=\Ar_m(D_n)(1)$ so the pushout $F(1)$ is the identity. 
Since $\Ar_m(0)$ gives no relations on polyhoms we have $Q=P\otimes \Ar_m(D_n)$ for polyhoms. Since $D_n$ is exact $F$ is a quasi-isomorphism on polyhoms.
\end{proof}

\begin{defin}
The class of morphisms which are surjective on objects, surjective on polyhoms and quasi-isomorphisms on polyhom is denoted by \bf{Surj}.
\end{defin}

\begin{lemma} \label{IinjSurj}
$I$-inj=\bf{Surj}.
\end{lemma}
\begin{proof}
Let $F:P \to Q$ be a morphism of operads fitting into a commutative diagram
\[ \xymatrix{\emptyset \ar[d] \ar[r] & P \ar[d]^F \\ \Acal \ar[r]^\alpha & Q}\]
A map $\alpha : \Acal \to P$ is equivalent to a choice of an object in $Q$. Hence, having the right lifting property with respect to $\emptyset \to \Acal$ is equivalent to $F$ being surjective on objects.
 
The rest of the proof is \cite[Proposition 2.3.5]{Ho}. We include it here for the readers convenience. Suppose that we have the right lifting property with respect to $S_{n-1} \to D_n$. A morphism of complexes
\[ \xymatrix{k\ar[d]^{\beta^ {n}} \ar[r]^\id & k \ar[d] \\ M^{n} \ar[r] & M^{n-1}}\]
is equivalent to a choice of an element in $M^{n}$. Likewise, a morphism $S_{n-1} \to M$ is given by an element $m$ in $M^{n-1}$ with $dm=0$. Thus, a diagram
\[ \xymatrix{S_{n-1}\ar[d] \ar[r]^f & X \ar[d]^p \\ D_n \ar[r]^g & Y}\]
is equivalent to the data $\{(y,x) \in Y_n \times Z_{n-1}X \mid px=dy\}$. A lift is equivalent to a $z \in X_n$ such that $pz=y$ and $dz=x$. Let $y \in Y_n$ with $dy=0$. Then $(y,0)$ defines a diagram so by the lifting property there exists $z \in X_n$ with $pz=y$ and $dz=0$. Hence, $p$ is surjective on homology. Let $x \in Z_nX$ with $px=dy$. Then $(y,x)$ defines a diagram so there  exists $z \in X_{n+1}$ with $pz=y$ and $dz=x$. Hence, $p$ is also injective on homology so it is a quasi-isomorphism. The morphism $p$ is also surjective. Indeed, let $y \in Y_n$. Then $dy$ is a cycle so there exists $x \in Z_{n-1}X$ such that $px=dy$. The pair $(y,x)$ defines a diagram so there exists a lift $z \in X_n$ with $pz=y$. 

Assume now that $p$ is surjective on polyhoms and quasi-isomorphism on polyhom. Given a diagram $(y,x)$ we need to define a lift $z \in X_n$ with $pz=y$ and $dz=x$. Since $p$ is surjective on polyhoms we have a short exact sequence
\[0 \to K \to X \to Y \to 0\]
Choose $w \in X_n$ with $pw=y$. Then $p(dw) = d(pw) = dy = px$ and $d(dw-x)=dx=0$ so $dw-x \in Z_{n-1}K$. Since $p$ is a quasi-isomorphism on polyhoms  $H_\bullet(K)=0$ so there exists $v \in K_n$ such that $dv=dw-x$. Set $z=w-v$. Then $pz=y$ and $dz=x$ as required.
\end{proof}

\begin{lemma}\label{JinjWSurj}
$J$-inj $\cap W =$\bf{Surj}.
\end{lemma}
\begin{proof}
As observed in lemma \ref{FibrationRLP} the right lifting property imply surjectivity on polyhoms. Since $\Acal \to \Hcal$ has no polyhoms a diagram
\[ \xymatrix{\Acal \ar[d] \ar[r]^\alpha & P \ar[d]^F \\ \Hcal \ar[r]^\beta & Q}\]
is equivalent to a diagram on the underlying categories
\[ \xymatrix{\Acal \ar[d] \ar[r]^\alpha & P(1) \ar[d]^{F(1)} \\ \Hcal \ar[r]^\beta & Q(1)}\]
The result now follows from \cite[Lemme 2.4]{Ta}.
\end{proof}

\begin{proof}[Proof of theorem \ref{ModelCatOperad}]
Condition (1) is clear. The lemmas \ref{Jcell}, \ref{IinjSurj} and \ref{JinjWSurj} proves conditions (4), (5) and (6) in theorem \ref{Ho}. We need to check that $J-\text{cell}\subseteq I-\text{cof}=\text{llp}(I-\text{inj)}$. Let $X \to Y$ be one of the morphisms in $J$ and $F:P \to Q$ a pushout. We need to define a lifting for any diagram with $G \in I$-inj.
\[ \xymatrix{X \ar[d] \ar[r] & P \ar[d]^F \ar[r] & A \ar[d]^G \\ Y \ar[r] & Q \ar[r] & B}\]
By the lemmas $I-\text{inj}=J-\text{inj} \cap W$ so there is a morphism $Y\to A$ making the diagram commutative. The desired morphism $Q \to A$ is obtained from the universal property of pushout.

For (2) and (3) let $\kappa$ be a cardinal, $\lambda$ a $\kappa$ filtered ordinal and
\[ X_0 \to X_1 \to \dots \to X_\beta \to \dots \]
a $\lambda$-sequence. Then
\begin{gather} 
\colim_{\beta < \lambda} \Hom(\emptyset, X_\beta)= \colim_{\beta < \lambda}  0=0=\Hom(\emptyset, \colim_{\beta < \lambda} X_\beta),\\
\colim_{\beta < \lambda} \Hom(\Ar_m(S_{n}), X_\beta)= \colim_{\beta < \lambda}  \{x_\beta \in X_\beta^{n} : dx_\beta=0\},\\
\Hom(\Ar_m(S_n), \colim_{\beta < \lambda} X_\beta)=\{x\in \colim_{\beta < \lambda} X_\beta^n : dx=0\}.
\end{gather}
Since morphisms are chain maps the two last colimits agree. In particular, the domains are small relative to $I$-cell. For $J$ we have
\begin{gather} 
\colim_{\beta < \lambda} \Hom(\Acal, X_\beta)= \colim_{\beta < \lambda}  \{ \text{Obj}(X_\beta) \}=\Hom(\Acal, \colim_{\beta < \lambda} X_\beta),\\
\colim_{\beta < \lambda} \Hom(\Ar_m(0), X_\beta)= \colim_{\beta < \lambda}  =0=\Hom(\Ar_m(0), \colim_{\beta < \lambda} X_\beta)\}.
\end{gather}
Thus, the domains are small relative to any set of morphisms.
\end{proof}

\section{Unital dg-operads}
In the following sections we define a model category of categories enriched over DG-Cat, and of strong homotopy functors between them. Since the enrichments have units given by identity functors we first define the subcategory of unital colored non-symmetric DG-operads and prove that it inherits a model category structure. Let $\E$ be a set of objects with a distinguished unit object $e$. Write $\Acal(0)=\oplus_t \Acal(e,t)$. Then we have an inclusion  $\Acal(0):=\oplus_{t \in \Ebold} \Acal(e,t) \hookrightarrow \oplus_{s,t \in \Ebold} \Acal(s,t)=\Acal(1)$.

\begin{defin}
A unital DG-operad $\Acal$ is a DG-operad satisfying that for any $n \in \N$ and tuple $m_1, \dots, m_n$ with $0 \leq m_i \leq 1$ the morphism of complexes
\[ A(n) \otimes_{A(1)^n} (A(m_1) \otimes \cdots \otimes A(m_n)) \to A(m_1 + \dots + m_n) \]
is an isomorphism.
\end{defin}

Writing the definition out in terms of colors we get
\[ A(s_1, \dots, s_n, t) \otimes (A(r_1, s_1) \otimes \cdots \otimes A(e,s_i) \otimes \cdots \otimes A(r_n,s_n)) \isomap A(r_1, \dots, \hat{e}, \dots, r_n,t) \]
Given $(s_1, \dots, s_n) \in \E^n$ we define the reduced expression $\text{red}(s, \dots, s_n)$ to be $(s_1, \dots, s_n)$ with all occurrences of $e$ removed; e.g. red$(s,e,t,r,e,e,v,u,e)=(s,t,r,v,u)$. Instead of $\red(s_1, \dots, s_n)=\emptyset$ we write $\red(s_1, \dots, s_n) =e$. In particular, for unitary operads $\Acal(s_1, \dots, s_n,t) \isomap \Acal(\text{red}(s_1, \dots, s_n),t).$

\begin{defin}
A morphism of unital colored operads $F : \Acal \to \Bcal$ is a morphism of operads $F: \Acal \to \Bcal$ satisfying that $F(e_\Acal)=e_\Bcal$. Denote the category of unital colored operads by DG-Oper$^u(k)$.
\end{defin}

The forgetful functor
\[ \text{For} : \text{DG-Oper}^u(k) \to \text{DG-Oper}(k) \]
has an adjoint given by
\[ \tilde{} : DG\text{-Oper}_{\Ebold} \to DG\text{-Oper}^{u}_{\Ebold \cup \{e\}} \]
with
\[ \tilde{\Acal}(s_1, \dots, s_n,t)=\begin{cases} \Acal(\red(s_1, \dots, s_n), t) & \text{if } t\neq e \text{ and } \red(s_1, \dots, s_n) \neq e \\ 0 & \text{if } t = e \text{ and } \red(s_1, \dots, s_n) \neq e \\
k &  \text{if } t = e \text{ and } \red(s_1, \dots, s_n) = e \\
0 & \text{if }  t \neq e \text{ and } \red(s_1, \dots, s_n) = e \end{cases} \]

\begin{lemma} \label{MapUnitalForget}
$\Map_u(\tilde{\Acal}, \Bcal)=\Map(\Acal, \For(\Bcal))$
\end{lemma}
\begin{proof}
Let $F \in \Map(\tilde{\Acal}, \Bcal)$. It is given by the data $F : \Ebold_\Acal \cup \{e_\Acal\} \to \Ebold_\Bcal$ and
\[ \tilde{\Acal}(s_1, \dots, s_n, t) \to \Bcal(F(s_1), \dots, F(s_n), F(t)) \]
for all $s_1, \dots, s_n, t \in \Ebold_\Acal$. Since $F(e_\Acal)=e_\Bcal$ it is given by $F |_{\Ebold_\Acal} : \Ebold_\Acal \to \Ebold_\Bcal$ and
\begin{align}
\Bcal(F(s_1), \dots, F(s_n), F(t)) & \iso \Bcal(\red(F(s_1), \dots, F(s_n)), F(t))\\
&=\Bcal(\red(F(\red(s_1, \dots, s_n))), F(t))\\
& \iso \Bcal(F(\red(s_1, \dots, s_n)), F(t)).
\end{align}
Hence, for $t\neq e$ and $\red(s_1, \dots, s_n) \neq e$ it is determined by $F|_\Acal \in \Map(\Acal, \For(\Bcal))$. In the cases $t=e, \red(s_1, \dots, s_n) \neq e$ and $  t \neq e, \red(s_1, \dots, s_n) = e$ the map is 0. In the only remaining case $t=e, \red(s_1, \dots, s_n) = e$ we have
\[ \tilde{\Acal}(e,\dots, e, e)=\tilde{\Acal}(e,e)=k \to \Bcal(e,e) \]
Since $F$ is in particular a functor between the categories $\tilde{\Acal}(1)$ and $\Bcal(1)$ it takes 1 to $\id \in \Bcal(e,e)$. Thus, we have shown that $F$ is completely determined by $F|_\Acal$. It is clear from the above that a non-unital map $\Acal \to \For(\Bcal)$ uniquely extends to a unital map $\tilde{\Acal} \to \Bcal$.
\end{proof}

In this category we define the generating fibrations and generating trivial fibrations to be $\tilde{\;}$ of the generators in the non-unital category.

\begin{defin}
\begin{enumerate}
\item $\tilde{I}$ is the set of maps $\{\tilde{\emptyset} \to \tilde{\Acal}, \tilde{\Ar}_m(S_{n-1})\to \tilde{\Ar}_m(D_n); m,n \in \N\}$.
\item $\tilde{J}$ is the set of maps $\{\tilde{\Acal} \to \tilde{\Hcal}, \tilde{\Ar}_m(0) \to \tilde{\Ar}_m(D_n) ; m,n \in \N \}$.
\end{enumerate}
\end{defin}

A map $F: X \to Y$ of non-unital operads upgrades to a map of unital operads $\tilde{F}: \tilde{X} \to \tilde{Y}$. By the lemma for a map of unital operads $G: P \to Q$ to have the right lifting property with respect to $\tilde{F}$ is equivalent to For$(G)$ to have the right lifting property with respect to $F$.

\[ \xymatrix{\tilde{X} \ar[d]_{\tilde{F}} \ar[r] & P \ar[d]^G \\ \tilde{Y} \ar[r] \ar@{-->}[ru] & Q}  \qquad \qquad \xymatrix{X \ar[d]_{F} \ar[r] & \text{For}(P) \ar[d]^{\text{For}(G)} \\ Y \ar[r] \ar@{-->}[ru] & \text{For}(Q)}\]

Hence, the unital weak equivalences, (resp. fibrations, resp. cofibrations are exactly those morphisms of unital operads which are weak equivalences (resp. fibrations, resp. cofibrations) in the non-unital model category.

\begin{thm}\label{ModStructUOper}
The category of unital DG-operads has a model structure.
\end{thm}
\begin{proof}
It follows directly from lemma \ref{IinjSurj} and \ref{JinjWSurj} together with the above remark that $\tilde{I}-\text{inj}=\tilde{W} \cap \tilde{J}-\text{inj}$. For the pushout we have $\tilde{Y} \times^{\tilde{X}} P\iso  Y \times^X P$.
Hence, it follows from lemma \ref{Jcell} that $\tilde{J}-cell \subseteq \tilde{W}$. These are the main steps in the proof. The rest is the same as in the proof of theorem \ref{ModelCatOperad}.
\end{proof}

\section{Bar construction for colored non-symmetric DG-operads}
In this section we define the bar construction for colored non-symmetric DG-operads. It is a modification of the bar construction for (non-colored) symmetric DG-operads in Getzler-Jones \cite{GJ} (which is again based on a bar construction by Ginzburg and Kapranov \cite{GK}) replace the abstract graph tree for symmetric operads with planar trees. Another good reference is \cite{LV}.

\subsection{Trees}
A (planar rooted) tree is a nonempty oriented, contractible planar graph without loops (oriented or not) such that there is a least one incoming edge and exactly one outgoing edge at each vertex. We allow edges to be bounded by a vertex at one end only. Such vertices are called external. All other edges are called internal. For a tree $T$ we denote the set of input edges by In$(T)$. For a vertex $v \in T$ we denote the number of incoming edges by In$(v)$. A tree with $n$ incoming edges is called an $n$-tree.

\subsection{The free DG-operad}
Let $P$ be an $\N$-collection. We define the $\N$-collection $F(P)$ 
\[ F(P)(n)=\bigoplus_{n\text{-tree } T} \bigotimes_{v \in T} P(\text{In}(v)). \]
Observe that there is a map $F(P) \circ F(P) \to F(P)$ (modifying the trees) giving a natural operad structure. This is called the free operad. Similarly, one also have a map $F(P) \to F(P) \circ F(P)$ giving a cooperad structure on $F(P)$. This is called the cofree cooperad and we denote it by $C(P)$. The free operad is free in the sense that the functor $F : \N\text{-coll} \to \text{DG-Oper}$ is right adjoint to the forgetful functor (see e.g. \cite[Cor. 1.11]{GJ}).

\subsection{Bar construction}
Let $P$ be a DG-operad. We write $\Sigma P$ for its suspension which shifts the degree of each complex by minus 1, i.e. $(\Sigma P)(n)^i= P(n)^{i-1}$ and $\delta_{\Sigma P(n)}(v)=(-1)^{|v|}\Sigma (\delta_{P(n)} v)$. The bar construction $\mathcal{B}(P)$ is the defined by adding a term to the differential on the cofree cooperad $C(P)$ given by the operad structure on $P$.

Let $T$ be a tree and $e$ be an internal edge from vertex $s$ to $t$ with $e$ being input number $k$. We define the tree $T\backslash e$ by contracting the edge $e$ and merging the vertices $s$ and $t$. The merged vertex in $T\backslash e$ is denoted by $t \cdot s$. Using the unit map $\eta : 1 \to P(1)$ and the operad structure we get a map
\[ P(\text{In}(t)) \otimes 1^{k} \otimes P(\text{In}(s)) \otimes 1^{\text{In}(t)\backslash e-k} \to P(\text{In}(t)) \otimes P(1)^{k} \otimes P(\text{In}(s)) \otimes P(1)^{\text{In}(t)\backslash e-k} \to P(\text{In}(t \cdot s))\]
This induces a map $\partial_e : (\Sigma P)(T) \to (\Sigma P)(T\backslash e)$ of degree -1. Let $\partial : C(\Sigma P) \to C(\Sigma P)$ be the sum of all the $\partial_e$ for all trees $T$ and internal edges $e$. As shown in \cite[Lemma 3.2.9]{GK} for any finite set $I$ and $V_i$ DG-vector spaces for $i \in I$ there is an isomorphism.
\begin{gather}
    \bigotimes_{i \in I} V_i[-1] \isomap (\bigotimes_{i \in I} V_i)[-|I|] \otimes \Lambda^{|I|} k^{|I|},\\
    v_{i_1} \otimes \dots \otimes v_{i_n} \mapsto v_{i_1} \otimes \dots \otimes v_{i_n} \otimes i_1 \wedge \dots \wedge i_n.
\end{gather}
Hence, if $e_1$ and $e_2$ are two distinct edges in $T$ then $\partial_{e_1} \partial_{e_2}+\partial_{e_2} \partial_{e_1}=0$ and so $\partial^2=0$. Since $\partial$ commutes with $\delta_P$ it follows that $\delta_P+\partial$ is a differential. In \cite[Prop. 2.2]{GJ} they also check that it is compatible with the operad structure.

\begin{defin}
Let $P$ be an augmented operad. The bar cooperad $\Bcal(P)$ is the cooperad $C(\Sigma P)$ with differential $\delta_{\Bcal(P)}=\delta_P + \partial$.
\end{defin}

Notice that bar for $P(1)$ is the same as bar for DG-categories (see Keller). A cooperad is connected if it has $k_\E$ in degree 0 and the rest is concentrated in strictly positive degree e.g. $\Bcal(P)$ is connected for any operad $P$. 
There is a cobar construction $\Bcal^*$ for connected cooperads. It is the operad $F(\Sigma^{-1}Q)$ with differential $\delta_{\Bcal^* Q}=\delta_Q + \partial^*$, where $\delta_Q$ is the internal differential induced by the one on $Q$ and $\partial^*$ is the differential defined by reversing all the arrows in the definition of $\partial$.

\begin{lemma}\cite [Thm. 2.17]{GJ}
The functors $\Bcal^*$ and $\Bcal$ forms an adjoint pair of functors between DG-operad and connected DG-cooperads.
\[ \Hom_{\text{DG-Oper}}(\Bcal^*(Q),P)=\Hom_{\text{DG-Coop}}(Q, \Bcal(P)). \]
\end{lemma}

The proof is for symmetric operads but the same proof works in our slightly modified setting.

\begin{prop}
For an operad $P$ we have a functorial cofibrant replacement given by $\Bcal^* \Bcal(P)$.
\end{prop}

We split the proof of the proposition into two lemmas.

\begin{lemma} \label{CobBarRes}
The counit $\Bcal^* \Bcal (P) \to P$ is a weak equivalence for any operad $P$.
\end{lemma}
\begin{proof}
This proof is an adaption of the proof of \cite[Theorem 3.2.16]{GK}. We include the proof in our setup for completeness and because it provides a model for the proof of a generalization of this statement (lemma \ref{B*Bres2Cat}). We first calculate $\Bcal^* \Bcal(P)$.
\begin{align}
    \Bcal^* \Bcal(P)(n) &=\bigoplus_{\substack{n \text{ tree } \\ T}} \bigotimes_{v \in T} \Bcal(P)[-1](\text{In}(v))\\
    &\iso \bigoplus_{\substack{n \text{ tree } \\ T}} \bigl( \bigotimes_{v \in T} \Bcal(P)(\text{In}(v)) \bigr)[-|T|] \otimes \det(T)\\
    &= \bigoplus_{\substack{n \text{ tree } \\ T}} \bigotimes_{v \in T}\bigoplus_{\substack{\text{In}(v) \\ \text{ tree } S}} \bigotimes_{w \in S} P[1](\text{In}(w)) [-|T|] \otimes \det(T)\\
    &= \bigoplus_{\substack{n \text{ tree } \\ T}} \bigotimes_{v \in T}\bigoplus_{\substack{\text{In}(v) \\ \text{ tree } S}} \bigl( \bigotimes_{w \in S} P(\text{In}(w)) \bigr)[-|T|+|S|] \otimes \det(T) \otimes \det(S)^*.
\end{align}
Observe that the two sums and tensor products means that we start with a tree $T$ and then we replace each vertex with a tree with the same valence. Hence, we get a direct sum where the same tree occurs many times corresponding to different starting trees and different trees inserted into it. If a tree $T'$ can be obtained for $T$ by contracting some (possibly empty) set of edges (i.e. as a starting tree it will give rise to a copy of $T$) then we write $T\geq T'$.

For $w \in T'$ we denote the subtree of $T$ contracted into $w$ by $T_w$. Observe that in the above formula $\det(S)^*$ cancels out part of $\det(T)$. Rewriting in terms of pairs $T \geq T'$ what remains is the part corresponding to the trees $T_w$ and the formula becomes.
\[\Bcal^* \Bcal(P)(n) \iso \bigoplus_{\substack{n \text{ tree } T\\ T \geq T'}} \bigl( \bigotimes_{v \in T} P(\text{In}(v))\bigr)[-|T|+|T'|] \bigotimes_{w \in T'} \det(T_w)\]
Notice that the term for $T=T'$ being the tree with one vertex and $n$ edges is $P(n)$. Hence, we have a surjective projection map $\Bcal^* \Bcal(P) \to P$ which is a morphism of operads. We need to show that the complexes for all other pairs are acyclic. The differential on $\Bcal^* \Bcal(P)$ has three terms $d_1+d_2+d_3$, where $d_1$ is the differential on $P$, $d_2$ is the part induced by the operad structure on $P$, and $d_3$ the part induced by the operad structure on $F(P[1])$, i.e. grafting of trees. Note that $\Bcal^* \Bcal(P)$ is the total complex of a double complex placed in the third quadrant with differentials $d_1+d_2$ and $d_3$. Notice that we have a bounded below exhaustive filtration compatible with the differentials given by
\[\Bcal^* \Bcal(P)(n)^m \iso \bigoplus_{\substack{n \text{ tree } T\\ T \geq T', |T|=m}} P(T)[-|T|+|T'|] \bigotimes_{w \in T'} \det(T_w).\]
The convergence of the spectral sequence of a bicomplex with a bounded below exhaustive filtration shows that it is enough to prove that the complex is acyclic with respect to $d_3$. Fix $n$ and $T$. Then the summand corresponding to $T$ is the product of $P(T)$ with the purely combinatorial complex with $-i$'th term 
\[C_i :=\bigoplus_{\substack{T \geq T'\\|T'|-|T|=i}} \bigotimes_{w \in T'} \det(T_w).\] 
Since $\det(T_w)$ is a one-dimensional vector space the $-i$'th term is $k^i$. A choice of $T'$ corresponds to a choice of edges in $T$ to contract and defines the $T_w$ which defines a way of splitting $T$ into up into $|T'|$ disjoint pieces. The differential on each $\det(T_w)$ is given by $i_1 \wedge \dots \wedge i_{|T_w|} \mapsto \sum_k i_1 \wedge \dots \wedge \widehat{i_k} \wedge \dots \wedge i_{|T_w|}$, where the hat indicates that the entry is left out. The differential corresponds to all possible ways of removing $|T'|$ edges keeping track of signs. This is exactly the face map in the chain complex of the $|T|$-simplex. Since a simplex is contractible the complex is acyclic except for in degree 0. Degree 0 is the part corresponding to $T'=T$ and so to $P(n)$. Hence, the map $\Bcal^* \Bcal(P) \to P$ is a quasi-isomorphism.
\end{proof}

\begin{lemma}\label{CobBarCofib}
The operad $\Bcal^* \Bcal(P)$ is cofibrant for any operad $P$.
\end{lemma}
\begin{proof}
Let $V$ be an $\N$-collection. Then $F(V)$ is cofibrant. Indeed, the free operad functor is left adjoint to the forgetful functor from operads to $\N$-collections. Hence, the left lifting property for $\emptyset \to \text{Free}(V)$ with respect to trivial fibrations in the model category of operads is equivalent to the left lifting property of $V$ with respect to termwise surjective quasi-isomorphisms of $\N$-collections. This is equivalent to a collection of lifts of surjective quasi-isomorphisms in the model category of complexes of vector spaces and here such a lift exists.

Let $P$ be an operad with a filtration $P_0=F(V) \subset P_1 \subset \dots \subset P_n \subset \dots$ such that $P=\bigcup P_i$ and at each step $P_i=P_{i-1} \langle x_i \rangle$ with $d x_i \in P_{i-1}$. Such an operad is called quasi-free. In this case we have a pushout diagram
\[\xymatrix{\text{Free}(k \, d x_i)  \ar[r] \ar[d] & P_{i-1} \ar[d] \\ \text{Free}(k \, x_i \to k \, d x_i) \ar[r] & P_i} \]
Notice that $\text{Free}(k \, d x_i) \to \text{Free}(k \, x_i \to k \, d x_i)$ is the generating cofibration $\text{Ar}_m(S_{n-1}) \to \text{Ar}_m(D_n)$ for some $n,m$ so all the maps $P_{i-1} \to P_i$ are cofibrations and $P$ is cofibrant. 

The last step in the proof is to show that $\Bcal^* \Bcal(P)$ is quasi-free. Consider the filtration from the proof of the previous lemma
\[\Bcal^* \Bcal(P)(n)^m \iso \bigoplus_{\substack{n \text{ tree } T\\ |T|=m,  \; T \geq T'}} P(T)[-|T|+|T'|] \bigotimes_{w \in T'} \det(T_w).\]
The differential maps $\Bcal^* \Bcal(P)(n)^m$ into $\bigcup_{k <m}\Bcal^* \Bcal(P)(n)^k$ so adding the generators of $P(T)$ one at a time for $T$ with $|T|=m$ and then proceeding to $T'$ with $|T'|=m+1$ one gets a filtration of the required form. Thus, $\Bcal^* \Bcal(P)$ is cofibrant.
\end{proof}

\section{2-Cat$_I$}
In this section we introduce the category replacing 2-Cat from Schanuel-Street \cite{SS} in the (non-unital) homotopy setting. The unital version is done in the next section.

\subsection{Definition of 2-Cat$_I$}\label{Def2CatI}
To see how to generalize our definition we first rewrite the tree language of DG-operads into another combinatorial model. Notice that a tree with one inner vertex is equivalent to a rooted oriented polygon
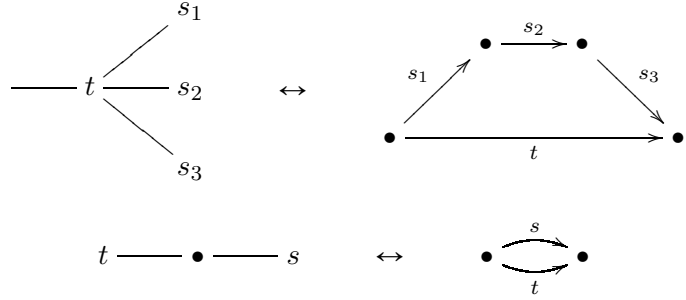
\[\vcenter{\vbox{\xymatrix@R=1.5em{& & s_1 \\  \ar@{-}[r] & t \ar@{-}[ur] \ar@{-}[r] \ar@{-}[dr] & s_2 \\ & & s_3
}}} \qquad \leftrightarrow \qquad \vcenter{\vbox{\xymatrix{& \bullet \ar[r]^{s_2} & \bullet \ar[rd]^{s_3} & \\
\bullet \ar[ur]^{s_1} \ar[rrr]_t & & & \bullet}}}
\]
\[ \xymatrix{t \ar@{-}[r] &\bullet \ar@{-}[r] & s} \qquad \leftrightarrow \qquad \xymatrix{\bullet \ar@/^/[r]^s \ar@/_/[r]_t & \bullet} \]
This extends to a bijection between marked trees and rooted oriented colored polygons with cellular decomposition. These are called chord diagrams and we denote them by Diag.
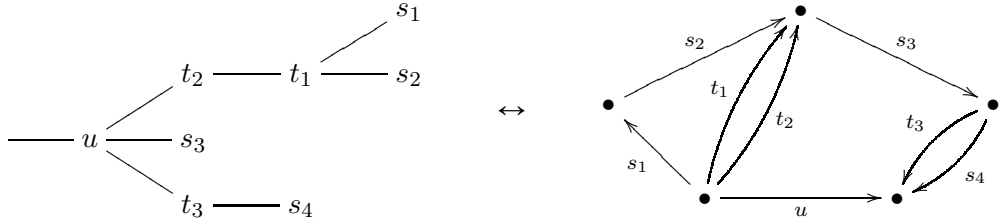
\[ \vcenter{\vbox{\xymatrix@R=1em{& & & & s_1\\ & & t_2 \ar@{-}[r]& t_1 \ar@{-}[ur] \ar@{-}[r] & s_2 \\ \ar@{-}[r]& u \ar@{-}[ru] \ar@{-}[r] \ar@{-}[rd] & s_3 &\\ & & t_3 \ar@{-}[r] & s_4 & }}} \qquad \leftrightarrow \qquad \vcenter{\vbox{\xymatrix{& & \bullet \ar[rrd]^{s_3} & & \\
\bullet \ar[rru]^{s_2} & & & & \bullet \ar@/^/[dl]^{s_4} \ar@/_/[dl]_{t_3} \\
& \bullet \ar[ul]^{s_1} \ar@/^/[uur]^{t_1} \ar@/_/[uur]_{t_2} \ar[rr]_u & & \bullet &}}} \]
With this description we have a natural generalization by also labeling the vertices of the polygon by a finite set $I$. Let $\E=\{\E_{ij}\}_{i,j \in I}$ be a collection of sets. We now rewrite the above diagram picture into description similar to the one for DG-operads. One may think of elements of $I$ as 0-morphisms and elements of $\E$ as 1-morphisms. Taking source or target defines two maps $s,t: \E \to I$. Given a diagram with a cell decomposition we consider paths different paths in it, i.e. $n$-step ways to go from one point to another along the arrows in the diagram.
\[ \text{Path}_{\E,I}(n) :=\E \times_I \cdots \times_I \E \qquad \text{$n$ factors}.\]
The source and target maps extend to Path$_{\E,I}$. Such a path picks out a cell in the diagram. 
\[\text{Cell}_{\E,I}=\bigsqcup_n \text{Cell}_{\E,I}(n)=\bigsqcup_{n} \text{Path}_{\E,I}(n) \times_{I \times I} \E.\]
The set of all diagrams (with coloring) is denoted by Diag$_{\E,I}$. The set of all possible cellular decompositions (with coloring) of a cell is denoted by dec. We have a map Diag$_{\E,I} \to \text{Cell}_{\E,I}$ given by taking all the cells occuring in the diagram (but forgetting how they are glued together). We have the projection maps top$:\text{Cell}_{\E,I} \to \text{Path}_{\E,I}$ and bot$:\text{Cell}_{\E,I} \to \E$. For a rooted oriented polygon cell $\nabla$ the edges bot$(\nabla)$ is the one such that $\nabla$ is to the left, e.g. in the picture above bot$(t_1,t_2)=t_2$ and bot$(t_3,s_4)=t_3$. The other edges are denoted by top$(\nabla)$. Concatenation of paths gives a natural structure of an operad in sets
\begin{gather}
\text{Cell}^n \times_{\text{Path}^n} (\text{Cell}^{m_1} \times_I \dots \times_I \text{Cell}^{m_n})\to \text{Cell}^{m_1+ \dots + m_n},\\
(t_1, \dots, t_n;t), (s_1; t_1), \dots, (s_n;t_n) \mapsto (s_1, \cdots, s_n; t).
\end{gather}
This can be rewritten on the level of rings. Set
\begin{gather}
k_I:=\bigoplus_{r \in \E} k e_r,  \qquad k_\E:=\bigoplus_{i,j \in I} k_{\E_{i,j}}
 \in k_I\text{-mod-}k_I,\\
\Ocal({\text{Path}^n}) :=k_\E \otimes_{k_I} \dots \otimes_{k_I} k_\E.
\end{gather}
Set $\Ocal(\text{Cell}^n)=\Ocal(\text{Path}^n) \otimes_{k_I \times k_I} k_\E$. Pulling back along the concatenation map gives a morphism of rings
\begin{align} \label{OCell}
\Ocal(\text{Cell}^{m_1+ \dots + m_n}) \to  \Ocal(\text{Cell}^n) \otimes_{\Ocal(\text{Path}^n)} (\Ocal(\text{Cell}^{m_1}) \otimes_{k_I} \dots \otimes_{k_I} \Ocal(\text{Cell}^{m_n}))
\end{align}

\begin{defin}\label{NseqIE}
The category $\N \text{-seq}_{I, \E}$ has objects collections of complexes $\{A(n) \mid n \geq 1\}$ with 
\[A(n)= \quad \smashoperator{\bigsqcup_{\substack{i,j \in I, t \in \E_{ij} \\ s \in \text{Path}^n(i,j)}}} \;A_{i,j}(s;t) \in \Ocal(\text{Cell}^n)\text{-mod}.\]
 This category has a product $\odot$ given by
\[ (A \odot B)(m) := \quad \smashoperator{\bigoplus_{\substack{n \in \N, \\ m_1 + \dots m_n=m}}} \; B(n) \otimes_{\Ocal({\text{Path}^n})} (A(m_1) \boxtimes_{k_I} \dots \boxtimes_{k_I} A(m_n)), \]
where $A(m_1) \boxtimes_{k_I} \dots \boxtimes_{k_I} A(m_n) \in \Ocal({\text{Path}^{m_1 + \dots +m_n}})$-mod-$\Ocal({\text{Path}^n})$ with the two the left and right $k_I$ module structures coinciding.
\end{defin}

This can be written out in coordinates. Let $s_n, \dots, s_1 \in \text{Path}^n(i,j)$ with $s_k \in \E_{i_{k-1}, i_k}$ with $i_0=i$ and $i_n=j$.
\begin{align}
(A \odot B)_{i,j}(s_n, \dots, s_1;u)&=\quad \smash[b]{\smashoperator{\bigoplus_{\substack{m_1 + \dots +m_r=n\\t_k \in \E_{i_{k-1}, i_k}}}}} \quad B_{i,j}(t_r, \dots, t_1;u) \otimes(A_{i,i_1}(s_{m_1}, \dots, s_1;t_1) \boxtimes \dots \\
& \hspace{3cm} \boxtimes A_{i_{n-1},j}(s_n, \dots, s_{m_1 + \dots + m_{r-1}+1};t_r))
\end{align}

\begin{defin}\label{2CatI}
A category in 2-Cat$_I$ with 0-morphisms $I$ and 1-morphisms $\E$ is a unital associative algebra in $\N-\text{seq}_{I, \E}$ with
\[ A(n)  \otimes_{\Ocal({\text{Path}^n})} (A(m_1) \boxtimes_{k_I} \dots \boxtimes_{k_I} A(m_n)) \to A(m_1 + \dots +m_n) \]
a morphism of $\Ocal(\text{Cell}^{m_1 + \dots m_n})$-modules via \eqref{OCell}. 
\end{defin}

\begin{rem}
Notice that for $I$ being a one-point set 2-Cat$_{\{*\}}=$DG-oper.
\end{rem}

\begin{defin}
A morphism in 2-Cat$_I$ is the data $F=\{F_\E, F(n), n \geq 1\} : (I, \E_A, A) \to (I,\E_B, B)$. Here $F_\E : \E_A \to \E_B$ is a morphism. It induces a map of rings $\Ocal(\text{Cell}^n_A) \to \Ocal(\text{Cell}^n_B)$ and we require that this map is compatible with \eqref{OCell}. For $n \geq 1$ the $F(n)$ are morphisms of $\Ocal(\text{Cell}^n_A)$-modules $F(n) : A(n) \to B(n)$ satisfying composition.
\end{defin}

\subsection{Model category structure on 2-Cat$_I$}

Our goal in this section is to prove the following theorem 

\begin{thm} \label{2CatIModelStruct}
The category 2-Cat$_I$ has a cofibrantly generated model category structure with
\begin{enumerate}
\item A morphism $F: (I, \E_A, A) \to (I, \E_B, B)$ is a weak-equivalence if for all $n \in \N, i,j \in I$ and $s \in \operatorname{Cell}^n(i,j)$
	\begin{enumerate}
	\item The morphism $F : A_{ij}(s) \to B_{ij}(F(s))$ is a quasi-isomorphism.
	\item The functors $H^0(F(1)): H^0(A_{i,j}(1)) \to H^0(B_{i,j}(1))$ are equivalences of DG-categories.
	\end{enumerate}
\item A morphism $F: (I, \E_A, A) \to (I, \E_B, B)$ is a fibration if for all $n \in \N, i,j \in I$ and $s \in \operatorname{Cell}^n(i,j)$
	\begin{enumerate}
	\item The morphism $F : A_{ij}(s) \to B_{ij}(F(s))$ is termwise surjective.
	\item The functors $H^0(F(1)): H^0(A_{i,j}(1)) \to H^0(B_{i,j}(1))$ are isofibrations.
	\end{enumerate}
\end{enumerate}
\end{thm}

This is done in the same way as for colored DG-operads by defining generating cofibrations and generating trivial cofibrations. Fix $i, j \in I$ then there is a functor
\[ \underline{\enskip}^{ij} : \text{DG-oper}_\E \to 2\text{-Cat}_I \]
given by 
\begin{gather}
(\underline{\E}^{ij})_{kl}=\begin{cases} \E & i=k, j=l \\ \emptyset & \text{otherwise} \end{cases}, \qquad
(\underline{A}^{ij})_{kl}=\begin{cases} A & i=k, j=l \\ 0 & \text{otherwise} \end{cases}
\end{gather}

Given a complex $K$ and $s=(i_1 \stackrel{a_1}{\to} i_2 \stackrel{a_2}{\to} \cdots \stackrel{a_{n-1}}{\to} i_n, i_1 \stackrel{a}{\to}i_n)\in \text{Cell}_{i,j}$ we define Ar$^K_s$ to be the object in 2-Cat$_I$ with 
\begin{gather}
(\E_{\text{Ar}^K_s})_{kl}:=\begin{cases}\{a_1 : i_1 \to i_2, \dots, a_{n-1} : i_{n-1} \to i_n, a : i_1 \to i_n\} & k=1, l=j \\ \emptyset & \text{otherwise}\end{cases}\\
(\text{Ar}^K_s)_{kl}(t):=\begin{cases} K & k=i, l=j, t=s \\ 0 & \text{otherwise} \end{cases}
\end{gather}

The generating cofibrations and generating trivial cofibrations are the ones from colored DG-operads upgraded to 2-Cat$_I$.

\begin{defin}
\begin{enumerate}
\item $\Ical$ is the set of maps 
\[\{\emptyset \to \underline{\Acal}^{i,j}, \Ar_{s}^{S_{n-1}}\to \Ar_{s}^{D_n} \mid  i,j \in I, s \in \text{Cell}(i,j)\}.\]
\item $\Jcal$ is the set of maps 
\[\{\underline{\Acal}^{i,j} \to \underline{\Hcal}^{i,j}, \Ar_{s}^{0}\to \Ar_{s}^{D_n} \mid i,j \in I, s \in \text{Cell}(i,j)\}.\]
\end{enumerate}
\end{defin}

\begin{proof}[Proof of theorem \ref{2CatIModelStruct}]
We check the conditions in theorem \ref{Ho}. To prove (4) we notice that for a push-out diagram we get
\[ \vcenter{\xymatrix{\underline{\Acal}^{ij} \ar[d] \ar[r] & P \ar[d] \\ \underline{\Hcal}^{ij} \ar[r]^(0.75){\ulcorner} & Q}} \Rightarrow
\vcenter{\xymatrix{\Acal \ar[d] \ar[r] & P_{ij} \ar[d] \\ \Hcal \ar[r]^(0.75){\ulcorner}& Q_{ij}}} \wedge \vcenter{\xymatrix{0 \ar[d] \ar[r] & P_{kl} \ar[d] \\ 0 \ar[r]^(0.75){\ulcorner} & Q_{kl}}} \text{for } k\neq i, l \neq j.
\]
The rightmost diagram implies that $P_{kl} \to Q_{kl}$ is an isomorphism for $k\neq i, l \neq j$. The other is a push-out of colored DG-operads. Since the conditions for being a weak equivalence in 2-Cat$_I$ is termwise the same as for colored DG-operads (4) follows from lemma \ref{Jcell}. The arguments for pushouts for $\Ar_{s}^{0}\to \Ar_{s}^{D_n}$ is similar.

For (5) and (6) we observe that for existence of lifts
\[ \vcenter{\xymatrix{\underline{\Bcal}^{ij} \ar[d] \ar[r] & P \ar[d] \\ \underline{\Ccal}^{ij} \ar@{-->}[ur] \ar[r]& Q}} \Leftrightarrow
\vcenter{\xymatrix{\Bcal \ar[d] \ar[r] & P_{ij} \ar[d] \\ \Ccal \ar@{-->}[ur] \ar[r] & Q_{ij}}}, \qquad
\vcenter{\xymatrix{\Ar_{s_1}^{K_1} \ar[d] \ar[r]^a & P \ar[d]_F \\ \Ar_{s_2}^{K_2} \ar@{-->}[ur] \ar[r] & Q}} \Leftrightarrow
\vcenter{\xymatrix{K_1 \ar[d] \ar[r] & P_{ij}(a(s_1)) \ar[d] \\ K_2 \ar@{-->}[ur] \ar[r] & Q_{ij}(F(a(s_1)))}}
\]
where $\Bcal$ and $\Ccal$ are arbitrary DG-operads, $K_1, K_2$ complexes and $s_1, s_2$ cells. Hence, the requirement for lifting these in 2-Cat$_I$ is (since weak equivalences are also defined termwise) equivalent to the corresponding lifts for colored DG-operads. The conditions now follows from lemma \ref{IinjSurj} and \ref{JinjWSurj}. The proof of the rest of the conditions is identical to the corresponding conditions in the proof of theorem \ref{ModelCatOperad}.
\end{proof}

\subsection{Bar construction for 2-Cat$_I$}
For colored DG-operad the free operad/co-free cooperad is given by
\[F(P)(n)=\bigoplus_{\substack{T \text{ planar}\\ n \text{ tree}}}\bigoplus_{\substack{\text{colorings}\\ \text{of edges}}} \bigotimes_{\substack{\text{vertex }v \\ \text{of } T}} P(\text{In}(v);\text{Out}(v)),  \]
where In denotes the colors of the inputs at $v$ and Out denotes the color of the output. Let $A \in \N$-coll$_\E,I$. With the polygon combinatorial model described in section \ref{Def2CatI} the free operad construction has the following natural generalization to 2-Cat$_I$
\begin{align}
F(A)(n)&=\bigoplus_{\substack{\text{Cellular}\\ \text{decomposition} \\ \text{of } n\text{-gon} }}\bigoplus_{\substack{\text{markings of}\\ \text{polytope and} \\ \text{decomposition}}} \bigotimes_{\substack{\text{Cell } c \text{ in}\\ \text{decomposition}}} A(\text{top}(c);\text{bot}(c))\\
&=\bigoplus_{D \in \text{Diag}_{\E,I}(n)} \bigotimes_{c \in \text{Cell}_{\E,I}(D)} A(\text{top}(c); \text{bot}(c)).
\end{align}

\begin{lemma}
$F_{\text{Set}}(\operatorname{Cell}_{\E,I})=\operatorname{Diag}_{\E,I}$.
\end{lemma}

Just like for operads we have maps $F(A) \circ F(A) \to F(A)$ and $F(A) \to F(A) \circ F(A)$. 

\begin{prop}\label{FreeAdjoint}
The free functor $F: \N \operatorname{-seq}_{I} \to 2\operatorname{-Cat}_I$ is the right adjoint of the forgetful functor.
\end{prop}
\begin{proof}
The proof is almost identical to the corresponding proof for operads (see e.g. \cite[Cor. 1.11]{GJ}).
We need to prove that for any $V\in \N$-coll$_I$ and $A \in $2-Cat$_I$ every map $f: V \to A$ factors through a map $F(V) \to A$ in 2-Cat$_I$. It suffices to show that for every $A \in$ 2-Cat$_I$ there is a natural map $\mu: F(A) \to A$ in 2-Cat$_I$. Fix $D \in \text{Diag}_{\E,I}$ and a decomposition. Enumerate the cells in the chosen decomposition increasing from inside to outside and clockwise. E.g. in the example
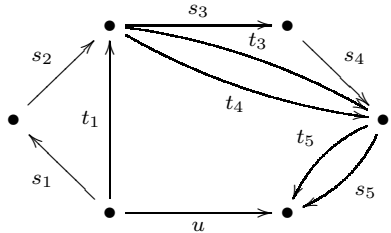
\[\vcenter{\vbox{\xymatrix{& \bullet \ar[rr]^{s_3} \ar@/^/[rrrd]^{t_3} \ar@/_/[rrrd]_{t_4}& & \bullet \ar[rd]^{s_4} & \\
\bullet \ar[ru]^{s_2} & & & & \bullet \ar@/^/[dl]^{s_5} \ar@/_/[dl]_{t_5} \\
& \bullet \ar[ul]^{s_1} \ar[uu]^{t_1} \ar[rr]_u & & \bullet &}}} \]
$c_1=(t_1,t_4,t_5,u), c_2=(s_1,s_2,t_1), c_3=(t_3,t_4), c_4=(t_5,s_5)$ and $c_5=(s_3,s_4,t_3)$. The desired map is obtained by iterated use of the operad structure on $A$ using this enumeration. E.g.
\vspace{-1em}
\begin{align}
&A(\text{top}(c_1); \text{bot}(c_1)) \otimes \Bigl( A(\text{top}(c_2); \text{bot}(c_2)) \boxtimes A(\text{top}(c_3); \text{bot}(c_3)) \boxtimes A(\text{top}(c_4); \text{bot}(c_4)) \Bigr)\\[-1em]
& \qquad \to A(s_1, s_2, t_3, s_5;u),\\[-0.3em]
& A(s_1, s_2, t_3, s_5;u) \otimes \bigl( 1 \boxtimes A(\text{top}(c_5); \text{bot}(c_5)) \boxtimes 1 \bigr) \to A(s_1, s_2, s_3, s_4, s_5;u).
\end{align}
Since the map is defined using the operad structure on $A$ the defined map $F(A) \to A$ is a morphism in 2-Cat$_I$.
\end{proof}

To define a Bar construction in 2-Cat$_I$ we only need to define the second differential. For operads this was defined by removing edges of trees. In the diagram language this corresponds to removing chords. Let $c$ and $c'$ be cells that can be glued together, i.e. the $k$th entry in top$(c)=$bot$c'$ for some $k$. Denote the cell obtained by gluing the cells and removing the gluing edge by $c \cdot c'$
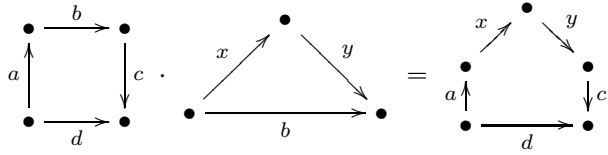
\[ \vcenter{\vbox{\xymatrix{\bullet \ar[r]^b & \bullet \ar[d]^c \\ \bullet \ar[u]^a \ar[r]_d & \bullet}}} \cdot \vcenter{\vbox{\xymatrix{& \bullet \ar[rd]^y & \\ \bullet \ar[ru]^x \ar[rr]_b & & \bullet}}}= \vcenter{\vbox{\xymatrix@=1em{& \bullet \ar[rd]^y & \\ \bullet \ar[ru]^x & & \bullet \ar[d]^c\\ \bullet \ar[u]^a \ar[rr]_d & &\bullet}}}\]
Such a pair $c,c'$ defines a differential $d_{c,c'}$ by the map
\[F(A)(\text{top}(c); \text{bot}(c)) \otimes \bigl[1^k \boxtimes F(A)(\text{top}(c'); \text{bot}(c')) \boxtimes 1^{|\text{top}(c)|-1-k}\bigr] \to F(P)(\text{top}(c\cdot c'); \text{bot}(c \cdot c'))\]
The bar construction $\Bcal(A)$ in 2-Cat$_I$ is taking free of $\Sigma A$
with the differential being the original differential plus the sum of all $d_{c,c'}$ for all such pairs $c,c'$. For the cobar $\Bcal^*(A)$ we take cofree of $\Sigma^{-1}A$ with differential being the sum of the original differential $d_A$ and a second term which is defined by reversing all arrows in the definition of the second term in the bar construction.

\begin{prop}
The functor $\Bcal$ is right adjoint to $\Bcal^*$.
\end{prop}
\begin{proof}
The proof is the same as for operads. To check that 
\[\Hom_{2\text{-Cat}_I}(A, \Bcal^*(C))=\Hom_{2\text{-coCat}_I}(\Bcal(A), C)\]
we notice that the underlying space for both $\Bcal$ and $\Bcal^*$ is free of something. By proposition \ref{FreeAdjoint} a map $f$ in both Hom spaces are determined by a map $\bar{f}: \Sigma A \to C$. Let $\mu: F(C) \to C$ be the map defined in the proof of proposition \ref{FreeAdjoint}. Notice that the restriction $\mu: C \odot C \to C$ is exactly $\delta$. The corresponding map $\Delta : A \to A \odot A$ in 2-coCat$_I$ is $\delta^*$. The condition for $\bar{f}$ to induce a map in 2-Cat$_I$ is compatibility with the differentials, i.e. the following diagram is commutative
\[\xymatrix@C=4em{\Sigma A \ar[d]^{d_A +\delta^*} \ar[r]^{\bar{f}} & C \ar[d]^{d_C} \\ \Sigma A \oplus (\Sigma A \odot \Sigma A) \ar[r]^(0.65){f + \delta (\bar{f} \odot \bar{f})} & C}\]
This condition; $f \circ d_A-d_C \circ \bar{f}+\Delta (f \odot f) \mu=0$ is the Maurer-Cartan equation for $\bar{f}$ and maps satisfying it are called twisting cochains. The condition for $\bar{f}$ to induce a map in $\Hom_{2\text{-coCat}_I}(\Bcal(A), C)$ is the same diagram.
\end{proof}

As in the DG-operad case $\Bcal^* \Bcal(A)$ is a cofibrant replacement of $A$.

\begin{lemma}\label{B*Bres2Cat}
The counit $\Bcal^* \Bcal (A) \to A$ is a weak equivalence for any $A \in 2\text{-Cat}_I$.
\end{lemma}
\begin{proof}
The proof is similar to the proof for DG-operads in lemma \ref{CobBarRes} but using the diagram combinetyatorics. We calculate $\Bcal^* \Bcal(A)(d)$ for a cell $\nabla$ using the same notation $\det$ as in the proof of the DG-operad lemma.
\[\Bcal^* \Bcal(A)(\nabla)=\bigoplus_{D \in \text{Dec}(\nabla)} \bigotimes_{d \in \text{Cell}(D)} \bigoplus_{C \in \text{Dec}(d)} \bigotimes_{c \in \text{Cell}(C)} A(c) \otimes \det(C)^* \otimes \det(D)[-|D|+|C|].
\]
As in the proof of that lemma we notice that a $c$ cell in $C$ also occurs in Cell$(D)$ for some decomposition $D$. Write $D \geq D'$ if the decomposition $D'$ is obtained from $D$ by removing a number of chords. Hence, each $A(d)$ occurs once for each $D \geq D'$ and we get
\[ \Bcal^* \Bcal(A)(\nabla)=\bigoplus_{\substack{D \in \text{dec}(\nabla)\\ D \geq D'}} \bigotimes_{c \in \text{Cell}(D)} A(c)[-|D|+|D'|] \bigotimes_{x \in \text{Cell}(D')}\det(D_x).\]
The subdiagrams $D_x$ of $D$ corresponds to the $T_w$ in the tree combinatorics. The only difference between the tree combinatorics for DG-operads and the diagram combinatorics is that with diagrams the number of possible decompositions depends on the labelling of the vertices. Looking at the above formula we see that once we fix a decomposition $D$ the rest is exactly the same for trees. In particular, the combinatorial complex $\otimes_x \det(D_x)$ is the same as the complex $\otimes_w \det(T_w)$ from lemma \ref{CobBarRes}. We already proved that this complex is acyclic except for in degree 0 (which corresponds to $D=D'$) where it is $k$. This finishes the proof.
\end{proof}

\begin{lemma}
For any $A \in$ 2-Cat$_I$ the object $\Bcal^* \Bcal (A)$ is cofibrant.
\end{lemma}
\begin{proof}
In the proof of the DG-operad version lemma \ref{CobBarCofib} we proved that quasi-free DG-operads are cofibrant. This directly generalizes to 2-Cat$_I$. Hence, all we need is to do is to find an exhaustive filtration on which the differential strictly lowers the degree and for which degree 0 is free. For a diagram $D$ let $|D|$ denote the number of cells in $D$. A such filtration is given by
\[ \Bcal^* \Bcal(A)(\nabla)^m=\bigoplus_{\substack{D \in \text{dec}(\nabla)\\ D \geq D', |D|=m}} \bigotimes_{c \in \text{Cell}(D)} A(c)[-|D|+|D'|] \bigotimes_{x \in \text{Cell}(D')}\det(D_x).\]
\end{proof}

\subsection{Unital 2-Cat$_I$}
As for colored DG-operads we can upgrade 2-Cat$_I$ to a unital version. Since $\E_{ij}$ is the set of 1-morphisms we have a distinguished unit object $\id_i$ in each $\E_{ii}$. Write $\Acal(0)=\bigsqcup_{i \in I, t \in \E_{ii}} \Acal_{ii}(\id_i,t)$. 

\begin{defin}
A unital 2-category $A$ is an object in 2-Cat$_I$ satisfying that for any $n \in \N$ and tuple $m_1, \dots, m_n$ with $0 \leq m_i \leq 1$ the morphism of complexes
\[ A(n) \boxtimes_{\Ocal(\text{Path}^n)} (A(m_1) \boxtimes \cdots \boxtimes A(m_n)) \to A(m_1 + \dots + m_n) \]
is an isomorphism. A unital morphism is a morphism $F$ in 2-Cat$_I$ such that $F(\id_i)=\id_i$ for all $i \in I$. We denote the unital category by 2-Cat$_I^u$.
\end{defin}

Note that this definition is compatible with the definition of unital DG-operads in the sense that for $A$ in 2-Cat$_I^u$ each $A_{ii}$ is a unital DG-operad. Analogously to the unital DG-operad case for $s=(i_1 \stackrel{a_1}{\to} i_2 \stackrel{a_2}{\to} \cdots \stackrel{a_{n-1}}{\to} i_n) \in \text{Path}_{i,j}$ its reduced expression is defined by removing all $\id_i$ for any $i \in I$; e.g. red$(s, \id_i, t, r, \id_j,  \id_j, v, u, \id_k)=(s, t, r, v, u)$. If $\red(s) = \emptyset$ then we write $\red(s)=\id$. In particular we have $A(s;t) \iso A(\red(s);t)$ in 2-Cat$_I^u$. There is a functor
\[ \bar{\enskip} : 2\text{-Cat}_I \to 2\text{-Cat}_I^u\]
For $i,j \in I$ with $i \neq j$ it is given by $\bar{\E}_{ij}:=\E_{ij}$ and $\bar{A}_{ij}(s;t) := A_{ij}(\red(s);t)$. For $i=j$ we set $\bar{\E}_{ii}:=\E_{ii} \cup \{\id_i\}$ and 
\[ \bar{A}_{ii}(s;t):= \begin{cases} A_{ii}(\red(s);t) & \text{if } \red(s) \neq \id_i \text{ and } t \neq \id_i \\
0 & \text{if } \red(s) = \id_i \text{ and } t \neq \id_i \\
k & \text{if } \red(d) = \id_i \text{ and } t = \id_i \\
0 & \text{if } \red(d) \neq \id_i \text{ and } t = \id_i 
\end{cases}\]
Note that with this definition the unital upgrade of the DG-operad $A_{ii}$ coincides with the upgrade $\tilde{A_{ii}}$ that was defined for DG-operads. In the same way as in the proof of lemma \ref{MapUnitalForget} one can show that $\Map_u(\bar{A}, B)\iso \Map(A, \For(B))$. This functor can be used to upgrade the generating cofibrations and generating trivial cofibrations from 2-Cat$_I$.

\begin{defin}
\begin{enumerate}
\item $\bar{\Ical}$ is the set of maps 
\[\{\bar{\emptyset} \to \bar{\underline{\Acal}}^{i,j}, \bar{\Ar}_{s}^{S_{n-1}}\to \bar{\Ar}_{s}^{D_n} \mid  i,j \in I, s \in \text{Cell}(i,j)\}.\]
\item $\bar{\Jcal}$ is the set of maps 
\[\{\bar{\underline{\Acal}}^{i,j} \to \bar{\underline{\Hcal}}^{i,j}, \bar{\Ar}_{s}^{0}\to \bar{\Ar}_{s}^{D_n} \mid i,j \in I, s \in \text{Cell}(i,j)\}.\]
\end{enumerate}
\end{defin}

The following theorem follows from theorem \ref{2CatIModelStruct} in the same way that theorem \ref{ModStructUOper} follows from theorem \ref{ModelCatOperad}.

\begin{thm}\label{2CatIUModelStruct}
The category 2-Cat$_I^u$ has a model category structure.
\end{thm}

\section{Homotopy adjunction}

In the previous section we defined 2-Cat$_I$ which is the homotopy version replacement of 2-categories. Now we need to define homotopy versions of the free adjunction category and Fun$_{A_1,A_2}$ in 2-Cat$_{\{0,1\}}$.

\subsection{Relaxed version of Fun$_{A_1,A_2}$}
Let $A_1$ and $A_2$ be cofibrant DG-categories. We need to define a homotopy version of Fun$_{A_1, A_2} \in 2-\text{Cat}_I$. For this we need to define the homotopy version DG-Fun$_\infty(A_i, A_j)$ for $i,j=1,2$. The objects are DG-functors $A_i \to A_j$ and the Hom spaces are coherent natural transformations between DG-functors (see e.g. \cite{Tam}, \cite{Fa1} and \cite{Fa2}).

\begin{defin}
Let $f,g: A \to B$ be DG-functors. Consider the $A$-$A$ bimodule ${_f} B_g$ with ${_f} B_g(s;t):=B(f(s);g(t))$ and bimodule structure given by
\[\xymatrix@R=1em{ A(s_2;s_1) \otimes _f B_g(s_1; t_1) \otimes A(t_1;t_2) \ar[d] \\ B(f(s_2);f(s_1)) \otimes B(f(s_1);g(t_1)) \otimes B(g(t_1); g(t_2)) \ar[d] \\ B(f(s_2);g(t_2))} \]
Coherent natural transformations between $f$ and $g$ are defined as Hochschield cochains
\[\Coh(f,g):=HC^\bullet(A,{_f}B_g)=\Hom_{\text{Vect}(k)}(\Bcal(A), {_f} B_g).\]
The category with objects DG-functors and morphisms coherent natural transformations is denoted by DGFun$_\infty(A,B)$.
\end{defin}

In the definition it is not clear how to compose coherent natural transformations. For this we use the following lemma

\begin{rem}\label{CohRewrite}
In the notation from the definition. Define the $A$-module ${_f} B$ with ${_f} B(s)=B(f(s))$ and the same left $A$-module structure as ${_f}B_g$. Then
\[\Coh(f,g)=\Hom_{\text{mod-}B, \Bcal(A)-\operatorname{comod}}(\Bcal(A) \otimes {_f}B, \Bcal(A) \otimes {_g}B).\]
\end{rem}
Indeed, observe that as an $A$-$A$ bimodule
\[{_f}B_g=\Hom_{\text{mod-}B}({_f}B,{_g}B). \]
$\phi \in \Hom_{\text{mod-}B}(B(f(s)),B(g(t)))=\Hom_{\text{mod-}B}(\bigoplus_b B(f(s),b),\bigotimes_{b'}B(g(t),b'))$ is determined my $\phi(\id_{f(s)}) \in B(g(t))$ and
\begin{align}
d(\phi)(b) &=d_B(\phi(\id) b)-(-1)^{|\phi(\id)|} \phi(\id) d_B b\\
&=d_B(\phi(\id)) b+(-1)^{|\phi(\id)|} \phi(\id) d_B b-(-1)^{|\phi(\id)|} \phi(\id) d_B b\\
&=d_B(\phi(\id)) b.
\end{align}
The $A$-$A$-bimodule structure on ${_f}B_g$ clearly coincides with the $A$-$A$-bimodule structure on $\Hom_{\text{mod-}B}({_f}B,{_g}B)$ with the left action given by precomposing with the action on ${_f}B$ and the right action given by postcomposing with the action on ${_g}B$. By $\Hom$-tensor adjunction we get
\begin{align}
\Coh(f,g)&\iso \Hom_{\text{mod-}B}(\Bcal(A) \otimes {_f}B,{_g}B)\\
& \iso \Hom_{\text{mod-}B, \Bcal(A)\text{-comod}}(\Bcal(A) \otimes {_f}B,\Bcal(A) \otimes {_g}B).
\end{align}

The composition in mod-$B$, $\Bcal(A)$-comod defines a composition DG-functor of several variables
\[\circ_n : \text{DGFun}_\infty(A_n, A_{n-1}) \times \text{DGFun}_\infty(A_{n-1}, A_{n-2}) \times \cdots \times \text{DGFun}_\infty(A_2, A_1) \to \text{DGFun}_\infty(A_n, A_1) \]
Notice that it is strictly associative, i.e. for any $m_1, \dots, m_n$ it satisfies
\[ \circ_n (\circ_{m_1}, \dots, \circ_{m_n})=\circ_{m_1 + \dots + m_n}. \]

To define an object in 2-Cat$_I$ we need to associate a complex of vector spaces to each cell. To the cell $f_n, \dots, f_1, g$ we associate
\[ \text{DGFun}_{\infty(A_1, A_2)}(f_n, \dots, f_1;g):=\Coh(f_n \circ \dots \circ f_1,g). \]
The composition is a polyfunctor, i.e. a DG-functor in each variable. Hence, it provides a map
\[\Coh(f_{m_1}^1 \circ \dots \circ f_1^1,g_1) \boxtimes \dots \boxtimes \Coh(f_{m_n}^n \circ \dots \circ f_1^n,g_n) \to \Coh(f_{m_n}^n \circ \dots \circ f_1^1, g_n \circ \dots \circ g_1)  \]
Using this we obtain a map
\[ \begin{tikzcd}[row sep=1em]
\Coh(g_n \circ \dots \circ g_1,h) \otimes \bigl( \Coh(f_{m_1}^1 \circ \dots \circ f_1^1,g_1) \boxtimes \dots \boxtimes \Coh(f_{m_n}^n \circ \dots \circ f_1^n,g_n) \bigr)\arrow{d}\\
\Coh(g_n \circ \dots \circ g_1,h) \otimes \Coh(f_{m_n}^n \circ \dots \circ f_1^1, g_n \circ \dots \circ g_1) \arrow{d}\\
\Coh(f_{m_n}^n \circ \dots \circ f_1^1, h).
\end{tikzcd}\]
This makes DGFun$_{\infty(A_1,A_2)}$ into an element in 2-Cat$_{\{0,1\}}$. In the same way, 2-Cat$_{0}$ is DG-Oper so given a DG-category A, DG-Fun$_\infty(A,A)$ naturally becomes an object.

\subsection{Homotopy adjunction}
Recall from definition \ref{FreeAdj} that the free adjunction 2-category Adj has Adj$_{0,0}=\Delta_k$, Adj$_{0,1}=\lrhom_k$, Adj$_{1,1}=\nabla_k$ and Adj$_{1,0}=\rrhom_k$. Notice that a composition of order preserving maps is order preserving and a composition of maps preserving first/last element preserves first/last element. Hence, composition gives maps
\[ \text{Adj}_{i,j}([n];[\ell]) \otimes \text{Adj}([\ell];[m]) \to \text{Adj}([n];[m])\]
These maps makes Adj an object in 2-Cat$_{\{0,1\}}$. 

\begin{defin}
Let $A$ be a cofibrant DG-category. A homotopy monad is an $A_\infty$-morphism $\Bcal^* \Bcal(Delta_k) \to \text{DGFun}_\infty(A,A)$.
\end{defin}

We now have all the pieces required to define homotopy adjunction.

\begin{defin}\label{HomAdjDef}
Let $A_1,A_2$ be cofibrant DG-categories.
A homotopy adjunction is an A$_\infty$-morphism $\Bcal^* \Bcal(\text{Adj}) \to \text{DGFun}_\infty(A_1,A_2)$.
\end{defin}

Notice that the data of an $A_\infty$-morphism in 2-Cat$_{\{0,1\}}$ in particular defines a homotopy monad
\begin{align}
\Fcal_{0,0}: \Bcal^*\Bcal(\Delta_k) \to \text{DGFun}_\infty(A_1,A_1).
\end{align}
The additional data is supposed to define in particular homotopy actions of the monad in the spirit of remark \ref{StreetRem}. We plan to study the obtained structure in detail in our next paper.


\begin{thebibliography}{Ho}

\bibitem[AL]{AL}
R. Anno and T. Logvinenko, \textit{Bar category of modules and homotopy adjunction for tensor functors},  arXiv:1612.09530.

\bibitem[Ca]{Ca}
G. Caviglia, \textit{A model structure for enriched coloured Operads}

\bibitem[CM]{CM}
D-C. Cisinski and I. Moerdijk, \textit{Dendroidal sets and simplicial operads}, Jour. of Topology, 1-52 (2013).

\bibitem[Fa1]{Fa1}
G. Faonte, \textit{Simplicial nerve of an $\Acal_\infty$-category}. Theory Appl. Categ. 32 (2017), Paper No. 2, 31-52.

\bibitem[Fa2]{Fa2} Faonte, Giovanni$\Acal_\infty$-functors and homotopy theory of dg-categories. J. Noncommut. Geom. 11 (2017), no. 3, 957-1000.

\bibitem[GJ]{GJ}
E. Getzler and J.D.S. Jones, \textit{Operads, homotopy algebra, and iterated integrals for double loop spaces}. arXiv:hep-th/9403055.

\bibitem[GK]{GK} 
V. Ginzburg, M. Kapranov, \textit{Koszul duality for Operads}. Duke Math. J. 76 (1994), no. 1, 203-272.

\bibitem[Ho]{Ho}
M. Hovey, \textit{Model categories}, Mathematical Surveys and Monographs 63, AMS, Providence, RI (1999).

\bibitem[LV]{LV}
J-L. Loday, B. Vallette, \textit{Algebraic Operads}.

\bibitem[Lu]{Lu} J. Lurie, \textit{Higher algebra}. Available at http://www.math.harvard.edu/~lurie/papers/HA.pdf.

\bibitem[SS]{SS}
S. Schanuel, R. Street, \textit{The free adjunction}, Chaiers de topologie et g\'eom\'etrie diff\'erentielle cat\'egoriques, Tome 27, n$^\circ$ 1 (1986), p. 81-83.

\bibitem[Sh]{Sh}
B. Shoikhet, \textit{On the twisted tensor product of small DG categories},
arXiv:1803.01191v3.

\bibitem[Ta]{Ta}
G. Tabuada, \textit{Une structure de cat\'{e}gorie de mod\`{e}les de Quillen sur la cat\'{e}gorie des dg-cat\'{e}gories}, C. R. Acad. Sci. Paris, Ser. I 340, 15-19 (2005).

\bibitem[Tam]{Tam}
D. Tamarkin, \textit{What do dg-categories form?}, Compos. Math. 143 (2007), no. 5, 1335-1358. 

\end{thebibliography}
\end{document}